\documentclass[11pt]{article}
\usepackage[latin1]{inputenc}
\usepackage{amsmath,amsthm,amssymb}
\usepackage{amsfonts}
\usepackage{amsmath,amsthm,amssymb,amscd}
\usepackage{latexsym}
\usepackage{color}
\usepackage{graphicx}
\usepackage{mathrsfs}
\usepackage{cite}
\usepackage{color,enumitem,graphicx}
\usepackage[colorlinks=true,urlcolor=black,
citecolor=black,linkcolor=black,linktocpage,pdfpagelabels,
bookmarksnumbered,bookmarksopen]{hyperref}

\makeatletter \@addtoreset{equation}{section} \makeatother

\newtheorem{theorem}{Theorem}[section]

\newtheorem{proposition}{Proposition}[section]
\newtheorem{lemma}{Lemma}[section]

\newtheorem{corollary}[theorem]{Corollary}

\allowdisplaybreaks

\everymath{\displaystyle}

\begin{document}

\title{Multi-peak solutions for a critical Choquard problem in dimension 2}

\author{Luca Battaglia\footnote{E-mail address:\tt{ luca.battaglia@uniroma3.it} (L. Battaglia).}\\
\footnotesize Dipartimento di Matematica e Fisica, Universit\`a degli Studi Roma Tre, Largo S. Leonardo Murialdo 1, 00146 Roma, Italy\\
Wenshan Luo\footnote{E-mail address:\tt{ mathluows@163.com} (W. Luo).}\\
\footnotesize School of Mathematics and Statistics, Southwest University, Chongqing, 400715, P.R. China}
\date{}
\maketitle

\begin{abstract}
We consider the following Choquard problem:
$$\left\{\begin{array}{ll}-\Delta u=\varepsilon^{4-\alpha}(I_\alpha\ast ke^u)ke^u&\text{in }\Omega,\\u=0&\text{on }\partial\Omega;\end{array}\right.$$
here, $\Omega\subset\mathbb R^2$ is smooth and bounded, $I_\alpha(x):=\frac1{|x|^\alpha}$ is the Riesz functional for $\alpha\in(0,2)$ and $k\in C^1\left(\overline\Omega\right)$ is strictly positive.\\
We show the existence of solutions which concentrate at a finite number of points $\xi_1,\dots,\xi_m\in\Omega$, for either $m=1$ or: any $m\in\mathbb N$, $\alpha>\frac12$ and $\Omega$ multiply connected.\\
Up to our knowledge, this is the first existence result for concentrating solutions for Choquard-type problems in planar domains.
\end{abstract}

\section{Introduction}\

We consider the following problem:
\begin{equation}\label{eq. original problem}
\boxed{\left\{\begin{array}{ll}-\Delta u(x)=\varepsilon^{4-\alpha}\left(\int_\Omega\frac{k(y)e^{u(y)}}{|x-y|^\alpha}dy\right)k(x)e^{u(x)}&x\in\Omega,\\u(x)=0&x\in\partial\Omega;\end{array}\right.}
\end{equation}
here, $\Omega\subset\mathbb R^2$ is a smooth bounded domain, $\alpha\in(0,2)$ is fixed and $k(x)\in C^1\left(\overline\Omega\right)$ satisfies $\inf_\Omega k>0$.\\

Choquard-type problems have been increasingly studied in recent years.\\
It has been first introduced by Pekar in \cite{pek} as a model in quantum mechanics. Then, Penrose in \cite{pen} proposed it in a self-gravitation model.\\
Several results have been obtained for the problem on the whole space $\mathbb R^N$, at first for power-type nonlinearities \cite{lieb,lio} and then for general nonlinearities \cite{mvs,bvs}.\\
The study of Choquard-type equations in bounded domains is more recent. Solutions have been obtained variationally for some sub-critical nonlinearities in \cite{gp,bc}; in the case of critical nonlinearities, concentrating solutions have been constructed in \cite{ghp} and some asymptotic analysis is made in \cite{ccyz}. However, it seems that no concentrating solutions have been constructed in the case of planar domains.\\

Problem \eqref{eq. original problem} is a natural nonlocal generalization, via the convolution with the Riesz potential $I_\alpha(x):=\frac1{|x|^\alpha}$, of the well-known Liouville equation on planar domains:
\begin{equation}\label{liouville}
\left\{\begin{array}{ll}-\Delta u(x)=\varepsilon^2k(x)e^{u(x)}&x\in\Omega,\\u(x)=0&x\in\partial\Omega.\end{array}\right.
\end{equation}
In fact, thanks to the Hardy-Littlewood-Sobolev inequality
$$\|I_\alpha\ast f\|_{\frac{2p}{2-\alpha p}}\le C\|f\|_p\qquad\forall p\in\left[1,\frac2\alpha\right),$$
exponential nonlinearities are critical for problem \eqref{eq. original problem} just as they are critical for \eqref{liouville}. On the other hand, since $c_\alpha I_\alpha\underset{\alpha\to2}\rightharpoonup\delta_0$ in the sense of measures, for some $c_\alpha\underset{\alpha\to2}\to0$, \eqref{liouville} can be seen, at least formally, as a limiting case of \eqref{eq. original problem} as $\alpha$ goes to $2$.\\

Problem \eqref{liouville} has been intensively studied for decades, concerning several aspects: solutions are obtained variationally in \cite{dja} and a fine blow-up analysis is given in \cite{bm,ls,mw}.\\
Using perturbative methods, in \cite{dkm,esp} there were constructed the so-called multi-peak solutions, namely solutions concentrating at a finite number of given points. The aforementioned blow-up results showed that, under a finite mass assumption, such are the only possible blow-up configuration.\\
The aim of the present paper is to construct multi-peak solutions to \eqref{eq. original problem} in the same spirit as \cite{dkm,esp}.\\

Solutions are constructed using a Lyapunov-Schmidt finite-dimensional reduction.\\
As a first step, we introduce an \emph{approximated solution} $U$ to problem \eqref{eq. original problem}. This is done by taking a superposition of solutions to the limiting problem
\begin{equation}\label{eq. limit problem}
-\Delta w(x)=\left(\int_{\mathbb R^2}\frac{e^{w(y)}}{|x-y|^\alpha}dy\right)e^{w(x)}\qquad x\in\mathbb R^2,
\end{equation}
suitably rescaled and translated.\\
Solutions to \eqref{eq. limit problem} have been classified in some recent works \cite{guopeng_2024,niu,gluck,yangyu} under mild extra assumptions; they actually coincide with the well-known family of \emph{bubbles}. A key property in this step is the invariance by translations and dilation of problem \eqref{eq. limit problem}. More details about this will be given later in the paper.\\
The approximate solutions is then shown to be indeed a good approximation for a solution to \eqref{eq. original problem}, in some suitable norm.\\

At this point, a finite-dimensional reduction is applied. In other words, we may apply a fixed point argument on the orthogonal complement of the kernel of the linearized operator, which then	 has crucial importance.\\
On the other hand, to get a true solution to problem \eqref{eq. original problem}, we need to solve a finite-dimensional equation. This is equivalent to finding the location of the centers $\xi_1,\dots,\xi_m$ of the $m$ copies ot the bubble. As a final step, we show that such points must be critical for the \emph{reduced functional}
\begin{equation}\label{phi}
\boxed{\Phi_m(\xi)=\sum_{j=1}^m\left(\frac8{4-\alpha}\log k(\xi_j)+H(\xi_j,\xi_j)\right)+\sum_{k\ne j}G(\xi_j,\xi_k),}
\end{equation}
which therefore plays a major role; here, $G(x,y)$ is the Green's function of $-\Delta$ on $\Omega$, that is the solution to
$$\left\{\begin{array}{ll}-\Delta G(x,y)=8\pi\delta_y(x)&x\in\Omega\\G(x,y)=0&x\in\partial\Omega,\end{array}\right.$$
and by $H(x,y)\in C^\infty\left(\Omega\setminus\{y\}\right)$ its regular part
\begin{equation}\label{regular part}
H(x,y):=G(x,y)-4\log\frac1{|x-y|}.
\end{equation}
Since we want the points $\xi_1,\dots,\xi_m$ not to coincide with each other, we look for the $m$-tuple $(\xi_1,\dots,\xi_m)$ in the Cartesian product $\Omega^m:=\Omega\times\dots \times\Omega$ minus its \emph{diagonal}:
$$\Delta:=\{(\xi_1,\cdots,\xi_m)\in\Omega^m:\,\xi_j=\xi_k\,\text{ for some }j\ne k\}.$$
The definition of $\Phi_m$ is very similar, although not identical to the reduced functional in \cite{dkm,esp}; therefore, it can be treated using results from such papers.\\

We are now in position to state the main result from this paper.

\begin{theorem}\label{th.1}
Assume $\alpha\in\left(\frac12,2\right)$ and $\Phi_m$ has a stable critical point $\xi=(\xi_1,\dots,\xi_m)\in\Omega^m\setminus\Delta$. Then for sufficiently small $\varepsilon>0$, problem \eqref{eq. original problem} admits a family of $m$-peak concentrating solutions.\\
More precisely, there exists $\varepsilon_0>0$ such that, for any $\varepsilon\in(0,\varepsilon_0)$, there exists a solution $u_\varepsilon$ to \eqref{eq. original problem} such that:
\begin{itemize}
\item $\varepsilon^{4-\alpha}\int_\Omega(I_\alpha\ast ke^{u_\varepsilon})ke^{u_\varepsilon}\underset{\varepsilon\to0}\to8m\pi;$
\item $u_\varepsilon\underset{\varepsilon\to0}\to-\infty$ in $L^\infty_{\mathrm{loc}}(\Omega\setminus\{\xi_1,\dots,\xi_m\})$.
\end{itemize}
\end{theorem}

We are able to get existence of results only if the parameter $\alpha$ is far enough from the limiting case $\alpha=0$, as the nonlocal interaction between the peaks seems to be too strong to get good enough estimates. However, we believe this is just a technical issue and we plan to study the remaining cases in a forthcoming paper.\\

To get existence of concentrating solutions to \eqref{eq. original problem} we therefore need some sufficient conditions for the existence of stable critical points for $\Phi_m$.\\
Finding critical points of $\Phi_m$ is, in general, quite a challenging task; however, we may exploit some previous results to get solutions in a wide range of cases. 

\begin{corollary}
Assume $\Omega$ is not simply connected. Then for $\varepsilon>0$ sufficiently small, problem \eqref{eq. original problem} admits a family of $m$-peak solutions.
\end{corollary}

\begin{proof}
By Theorem \eqref{th.1}, it is sufficient to prove that the functional $\Phi_m(\xi)$ has a stable a critical point. If $\Omega$ is not simply connected, then $\Phi_m(\xi)$ has a critical point by \cite[Theorem 1]{dkm}.
\end{proof}\

In \cite[Theorem 5.4]{esp} the authors showed that, for any $m\ge2$, $\phi_m$ may have critical points even on simply connected domains, known as dumb-bell domains. Roughly speaking, they consist in a $m$ separate simply-connected region joined by a sufficiently narrow strip; see \cite[Section 5]{esp} for a formal definition.\\
Thanks to Theorem \ref{th.1} we get existence of concentrating solutions also in this case.\\

In the case $m=1$ of $1$-peak solutions, we are able to cover the whole range of $\alpha$ by the very same argument. In fact, we do not have issues from the mutual interaction of multiple peaks.

\begin{theorem}\label{m=1}
Assume $\alpha\in(0,2)$ and $m=1$. Then for sufficiently small $\varepsilon>0$, problem \eqref{eq. original problem} admits a family of $1$-peak concentrating solutions, in the same sense as Theorem \ref{th.1}.
\end{theorem}

The plan of the paper is as follows.\\
In Section 2 we introduce our ansatz for the solution, which in Section 3 we show to be a sufficiently good approximation. In Section 4 we analyze the linearized operator; in Section 5 we study the nonlinear operator to which we apply the fixed-point argument. Finally, in Section 6 we perform the finite-dimensional reduction and prove the main results.

\section{Ansatz}\

In this section we will give an \emph{ansatz} for solutions to problem \eqref{eq. original problem}.\\

We define, for $\mu>0$ and $\xi\in\mathbb R^2$, the bubble
\begin{equation}\label{bubble form}
w_{\mu,\xi}(x)=\frac{4-\alpha}2\log\left(\frac{C_\alpha\mu}{\mu^2+|x-\xi|^2}\right),
\end{equation}
where 
$$C_\alpha:=\left(\frac{(2-\alpha)(4-\alpha)}\pi\right)^{\frac1{4-\alpha}}.$$
It can be explicitly verified that \eqref{bubble form} is a solution to \eqref{eq. limit problem}.\\
Actually, it was shown in the papers \cite{guopeng_2024,niu,gluck,yangyu} that all solutions $w\in L^1_\mathrm{loc}\left(\mathbb R^2\right)$ to \eqref{eq. limit problem} satisfying the finite-mass condition
$$\int_{\mathbb R^2}e^{\frac4{4-\alpha}w(x)}dx<+\infty$$
are indeed in the form \eqref{bubble form}.\\

We fix a positive integer $m\in\mathbb N$ and choose $m$ distinct points $\xi_1,\dots,\xi_m\in\Omega$ and $m$ positive numbers $\mu_1,\dots,\mu_m>0$. We define
$$u_j(x):=\frac{4-\alpha}2\log\left(\frac{C_\alpha\mu_j}{\left(\mu_j^2\varepsilon^2+|x-\xi_j|^2\right)k^{\frac2{4-\alpha}}(\xi_j)}\right)=w_j\left(\frac x\varepsilon\right)-(4-\alpha)\log\varepsilon-\log k(\xi_j),$$
where $w_j:=w_{\mu_j,\frac{\xi_j}\varepsilon}$ is a bubble in the form \eqref{bubble form}. Clearly $w_j$ solves
$$-\Delta u_j(x)=\varepsilon^{4-\alpha}\left(\int_{\mathbb R^2}\frac{k(\xi_j)e^{u_j(y)}}{|x-y|^\alpha}dy\right)k(\xi_j)e^{u_j(x)}\qquad\text{in }\mathbb R^2.$$\

Since we look for solutions $u\in H^1_0(\Omega)$ to \eqref{eq. original problem}, we need to add to each $u_j$ an extra term in order to get zero Dirichlet boundary conditions. More precisely, we define the function $H_j(x)$ as the solution to
$$\left\{\begin{array}{ll}-\Delta H_j(x)=0&x\in\Omega,\\H_j(x)=-w_j\left(\frac x\varepsilon\right)+(4-\alpha)\log\varepsilon+\log k(\xi_j)&x\in\partial\Omega.\end{array}\right.$$
We then define the approximate solution $U$ as
$$\boxed{U(x):=\sum_{j=1}^m(u_j(x)+H_j(x))\in H^1_0(\Omega).}$$\

Let us now deduce some basic yet crucial estimates for $U(x)$.\\
Thanks to the explicit expression
$$w_j\left(\frac x\varepsilon\right)=\frac{4-\alpha}2\log\left(\frac{C_\alpha\mu_j\varepsilon^2}{\mu_j^2\varepsilon^2+|x-\xi_j|^2}\right)$$
we can write, as $\varepsilon\to0$:
\begin{equation}\label{eq.estimate of Hj}
H_j(x)=\frac{4-\alpha}4H(x,\xi_j)-\frac{4-\alpha}2\log(C_\alpha\mu_j)+\log k(\xi_j)+O\left(\mu_j^2\varepsilon^2\right),
\end{equation}
uniformly in $x\in\partial\Omega$; in fact, from the definition \eqref{regular part} of $H$, $H(x,\xi_j)$ is harmonic and coincides with $-4\log\frac1{|x-\xi_j|}$ for $x\in\partial\Omega$. Using the maximum principle, the same estimates holds uniformly on $\overline\Omega$.\\
On the other hand, for $x\ne \xi_j$ and away from $\xi_j$ we have:
$$u_j(x)=\frac{4-\alpha}2\log(C_\alpha\mu_j)-\log k(\xi_j)-(4-\alpha)\log|x-\xi_j|+O\left(\mu_j^2\varepsilon^2\right),$$
hence, similarly as before,
\begin{equation}\label{eq.estimate of mu}
H_j(x)+u_j(x)=\frac{4-\alpha}4G(x,\xi_j)+O\left(\mu_j^2\varepsilon^2\right).
\end{equation}
In order that $u$ is a good approximation, we choose $\mu_j$ satisfying
\begin{equation}\label{eq. definition of mu}
\log(C_\alpha\mu_j)=\frac2{4-\alpha}\log k(\xi_j)+\frac12\left(H(\xi_j,\xi_j)+\sum_{k\ne j}G(\xi_j,\xi_k)\right);
\end{equation}
subsequent computations will justify this choice.\\

Throughout the whole paper, we will choose the points $\xi_1,\dots,\xi_m$ not collapsing on each other, neither approaching the boundary, as $\varepsilon$ goes to $0$; more precisely, we will fix some small $\tau>0$ such that
\begin{equation}\label{dist}
\mathrm{dist}(\xi_j,\partial\Omega)\ge2\tau\quad\forall j=1,\dots,m;\qquad|\xi_j-\xi_k|\ge3\tau,\quad\forall j\ne k.
\end{equation}
Under this choice, estimates \eqref{eq.estimate of Hj} and \eqref{eq.estimate of mu}, will hold uniformly on $\xi$; similarly, all the estimates throughout the paper will be meant to be uniform on $\xi$, unless otherwise specified. No further references will be made about this.\\
From \eqref{dist} we also deduce
$$\frac1C\le\mu_j\le C,\qquad\forall j\in\{1,\dots,m\},\qquad\qquad\text{for some }C>0.$$\

It will also be convenient to rescale the problem by a factor $\varepsilon$, which compensates the concentration of the bubbles $u_j$. One easily sees that $u$ satisfies problem \eqref{eq. original problem} if and only if
$$v(x):=u(\varepsilon x)+(4-\alpha)\log\varepsilon,$$
satisfies
\begin{equation}\label{eq. change original problem}
\left\{\begin{array}{ll}-\Delta v(x)=\left(\int_{\Omega_\varepsilon}\frac{k(\varepsilon y)e^{v(y)}}{|x-y|^\alpha}dy\right)k(\varepsilon x)e^{v(x)}&x\in\Omega_\varepsilon,\\v(x)=(4-\alpha)\log\varepsilon&x\in\partial\Omega_\varepsilon,
\end{array}\right.
\end{equation}
where
$$\Omega_\varepsilon:=\frac\Omega\varepsilon:=\left\{\frac x\varepsilon:\,x\in\Omega\right\}.$$
Clearly, we will choose as an approximate solution to \eqref{eq. change original problem} the function $V$ defined as:
\begin{equation}\label{V}
V(x):=U(\varepsilon x)+(4-\alpha)\log\varepsilon
\end{equation}
We will then look for solutions to \eqref{eq. change original problem} in the form:
$$v(x)=V(x)+\phi(x),$$
where $V$ is as in \eqref{V} and $\phi$ is a small perturbation.\\
It is easy to see that, in order to get a solution, $\phi$ must solve the nonlinear equation
\begin{equation}\label{probphi}
\boxed{\mathcal E-\mathcal L\phi-\mathcal N(\phi)=0.}
\end{equation}
$\mathcal E$ will be denoted as the \emph{error}, $\mathcal L\phi$ will be denoted as the \emph{linear term}, $\mathcal N(\phi)$ will be denoted as the \emph{nonlinear term}. Their definition is the following:
\begin{align}
\label{ELN}\mathcal E(x):=&\sum_{j=1}^m\left(\int_{\Omega_\varepsilon}\frac{e^{w_j(y)}}{|x-y|^\alpha}dy\right)e^{w_j(x)}-\left(\int_{\Omega_\varepsilon}\frac{k(\varepsilon y)e^{V(y)}}{|x-y|^\alpha}dy\right)k(\varepsilon x)e^{V(x)},\\
\nonumber\mathcal L\phi(x):=&\Delta\phi+\left(\int_{\Omega_\varepsilon}\frac{k(\varepsilon y)e^{V(y)}}{|x-y|^\alpha}dy\right)k(\varepsilon x)e^{V(x)}\phi(x)+\left(\int_{\Omega_\varepsilon}\frac{k(\varepsilon y)e^{V(y)}\phi(y)}{|x-y|^\alpha}dy\right)k(\varepsilon x)e^{V(x)},\\
\nonumber\mathcal N(\phi)(x):=&\left(\int_{\Omega_\varepsilon}\frac{k(\varepsilon y)e^{V(y)}e^{\phi(y)}}{|x-y|^\alpha}dy\right)k(\varepsilon x)e^{V(x)}e^{\phi(x)}-\left(\int_{\Omega_\varepsilon}\frac{k(\varepsilon y)e^{V(y)}}{|x-y|^\alpha}dy\right)k(\varepsilon x)e^{V(x)}\\
\nonumber&-\left(\int_{\Omega_\varepsilon}\frac{k(\varepsilon y)e^{V(y)}}{|x-y|^\alpha}dy\right)k(\varepsilon x)e^{V(x)}\phi(x)-\left(\int_{\Omega_\varepsilon}\frac{k(\varepsilon y)e^{V(y)}\phi(y)}{|x-y|^\alpha}dy\right)k(\varepsilon x)e^{V(x)}.
 \end{align}

\section{Approximation of the solution}\

This section will be devoted to estimate the error term $\mathcal E$ introduced in \eqref{ELN}.\\

In order for $V$ to be a good approximation for a solution to \eqref{eq. change original problem} we need $\mathcal E$ to be small, in some sense which will be specified in a moment; in fact, we point out that if $V$ were a solution to \eqref{eq. change original problem}, $\mathcal E$ would identically vanish.\\
To this purpose, we introduce a norm, inspired by \cite{dkm}, which will be used to estimate the size of $\mathcal E$, and later on also $\mathcal L\phi$ and $\mathcal N(\phi)$. It is basically a weighted $L^\infty$-norm with a weight resembling the bubbles. We define:
\begin{equation}\label{norm}
\|h\|_\ast:=\sup_{x\in\Omega_\varepsilon}\frac{|h(x)|}{\sum_{j=1}^m\frac1{\left(1+\left|x-\xi'_j\right|\right)^{\min\{3,4-\alpha\}}}+\varepsilon^2},
\end{equation}
where
$$\xi'_j:=\frac{\xi_j}\varepsilon,\qquad\forall j=1,\dots,m.$$
We remark that, since $\int_{\Omega_\varepsilon}\left(\sum_{j=1}^m\frac1{\left(1+\left|x-\xi'_j\right|\right)^{\min\{3,4-\alpha\}}}+\varepsilon^2\right)dx\le C$, with $C$ independent on $\varepsilon$, then for any $f,g\in L^\infty(\Omega_\varepsilon)$ we get the estimate
$$\left|\int_{\Omega_\varepsilon}fg\right|\le C\|f\|_\ast\|g\|_\infty.$$\

The rest of this section will consist in showing that the $\|\cdot\|_\ast$-norm of $\mathcal E$ goes to $0$ as $\varepsilon$ goes to $0$. More precisely, we have the following:

\begin{proposition}\label{le. estimate of E}
There exists $C>0$ such that, for any $\varepsilon>0$, we have
$$\|\mathcal E\|_\ast\le C\varepsilon^{\min\{\alpha,1\}}.$$
\end{proposition}

\begin{proof}
Take $\tau$ as in \eqref{dist} and define the $m+1$ disjoint sub-domains of $\Omega_\varepsilon$:
\begin{equation}\label{bj}
B_j:=\left\{x\in\Omega_\varepsilon:\,\left|x-\xi'_j\right|<\frac\tau\varepsilon\right\},\quad j=1,\dots,m;\qquad B_0:=\Omega_\varepsilon\setminus\bigcup_{j=1}^mB_j.
\end{equation}
We observe that, for any $j\in\{1,\dots,m\}$ and $y\in B_j$, from \eqref{eq.estimate of Hj}, \eqref{eq.estimate of mu} and \eqref{eq. definition of mu} we get
\begin{align}
\nonumber&k(\varepsilon y)e^{V(y)}\\
\nonumber=&\frac{k\left(\xi_j+\varepsilon\left(y-\xi'_j\right)\right)}{k(\xi_j)}e^{w_j(y)+\frac{4-\alpha}4\left(H_j\left(\xi_j+\varepsilon\left(y-\xi'_j\right),\xi_j\right)+\sum_{k\ne j}G\left(\xi_j+\varepsilon\left(y-\xi'_j\right),\xi_k\right)-H_j(\xi_j,\xi_j)-\sum_{k\ne j}G(\xi_j,\xi_k)\right)}\\
\label{eq. in ball}=&e^{w_j(y)}\left(1+O\left(\varepsilon\left|y-\xi'_j\right|\right)\right).
\end{align}
On the other hand, we also get the global estimates:
\begin{equation}\label{eq. out ball}
e^{w_j(y)}=O\left(\frac1{\left(1+\left|y-\xi'_j\right|\right)^{4-\alpha}}\right),\,j=1,\dots,m;\quad k(\varepsilon y)e^{V(y)}=O\left(\sum_{j=1}^m\frac1{\left(1+\left|y-\xi'_j\right|\right)^{4-\alpha}}\right).
\end{equation}
Recalling the definition of $\mathcal E$ in \eqref{ELN}, we have
\begin{align*}
\mathcal E(x)=&\sum_{j=1}^m\underbrace{\sum_{k=1}^m\left(\int_{B_j}\frac{e^{w_k(y)}}{|x-y|^\alpha}dy\right)e^{w_k(x)}}_{=:I_{1,j}(x)}-\sum_{j=1}^m\underbrace{\left(\int_{B_j}\frac{k(\varepsilon y)e^{V(y)}}{|x-y|^\alpha}dy\right)k(\varepsilon x)e^{V(x)}}_{=:I_{2,j}(x)}\\
&+\underbrace{\sum_{k=1}^m\left(\int_{B_0}\frac{e^{w_k(y)}}{|x-y|^\alpha}dy\right)e^{w_k(x)}}_{:=I_{1,0}(x)}-\underbrace{\left(\int_{B_0}\frac{k(\varepsilon y)e^{V(y)}}{|x-y|^\alpha}dy\right)k(\varepsilon x)e^{V(x)}}_{=:I_{2,0}(x)}.\\
\end{align*}
If $x\in B_j$, then
$$I_{1,j}(x)=\left(\int_{B_j}\frac{e^{w_j(y)}}{|x-y|^\alpha}dy\right)e^{w_j(x)}+O\left(\varepsilon^{6-\alpha}\right)$$
and
$$I_{2,j}(x)=\left(\int_{B_j}\frac{e^{w_j(y)}\left(1+O\left(\varepsilon\left|y-\xi'_j\right|\right)\right)}{|x-y|^\alpha}dy\right)e^{w_j(x)}\left(1+O\left(\varepsilon\left|x-\xi'_j\right|\right)\right).$$
We may then estimate their difference using the following consequences of the triangular inequalities:
\begin{align}
\nonumber|x-y|\le\frac{1+\left|x-\xi'_j\right|}2\qquad&\Rightarrow\qquad1+\left|y-\xi'_j\right|\ge\frac{1+\left|x-\xi'_j\right|}2\\
\nonumber|x-y|\le2\left(1+\left|x-\xi'_j\right|\right)\qquad&\Rightarrow\qquad1+\left|y-\xi'_j\right|\le3\left(1+\left|x-\xi'_j\right|\right)\\
\label{triang}|x-y|\ge2\left(1+\left|x-\xi'_j\right|\right)\qquad&\Rightarrow\qquad|x-y|\ge\frac23\left(1+\left|y-\xi'_j\right|\right);
\end{align}
in fact,
\begin{align}
\nonumber&I_{1,j}(x)-I_{2,j}(x)\\
\nonumber=&O\left(\left(\int_{B_j}\frac{\varepsilon\left(\left|y-\xi'_j\right|+\left|x-\xi'_j\right|\right)}{|x-y|^\alpha}e^{w_j(y)}dy\right)e^{w_j(x)}+\varepsilon^{6-\alpha}\right)\\
\nonumber=&O\left(\varepsilon\left(\int_{B_j}\frac{\left|y-\xi'_j\right|+\left|x-\xi'_j\right|}{|x-y|^\alpha\left(1+\left|y-\xi'_j\right|\right)^{4-\alpha}}dy\right)\frac1{\left(1+\left|x-\xi'_j\right|\right)^{4-\alpha}}+\varepsilon^{6-\alpha}\right)\\
\nonumber=&O\left(\varepsilon\left(\int_{\left\{|x-y|\le\frac{1+\left|x-\xi'_j\right|}2\right\}}\frac1{|x-y|^\alpha}dy\right)\frac1{\left(1+\left|x-\xi'_j\right|\right)^{7-2\alpha}}\right.\\
\nonumber&\left.+\varepsilon\left(\int_{\left\{\frac{1+\left|x-\xi'_j\right|}2<|x-y|\le2\left(1+\left|x-\xi'_j\right|\right)\right\}}\frac1{\left(1+\left|y-\xi'_j\right|\right)^{4-\alpha}}dy\right)\frac1{\left(1+\left|x-\xi'_j\right|\right)^3}\right.\\
\nonumber&\left.+\varepsilon\left(\int_{\left\{|x-y|>2\left(1+\left|x-\xi'_j\right|\right)\right\}}\frac1{\left(1+\left|y-\xi'_j\right|\right)^4}dy\right)\frac1{\left(1+\left|x-\xi'_j\right|\right)^{4-\alpha}}+\varepsilon^{6-\alpha}\right)\\
\nonumber=&O\left(\varepsilon\left(\frac1{\left(1+\left|x-\xi'_j\right|\right)^{5-\alpha}}+\frac1{\left(1+\left|x-\xi'_j\right|\right)^3}+\frac\varepsilon{\left(1+\left|x-\xi'_j\right|\right)^{4-\alpha}}+\varepsilon^{5-\alpha}\right)\right)\\
\label{i1j-i2j}=&O\left(\frac{\varepsilon}{\left(1+\left|x-\xi'_j\right|\right)^3}\right).
\end{align}
If $x\in B_l$ for some $l\ne j$, then $|x-y|\ge\frac\tau\varepsilon$ for any $y\in B_j$, hence from \eqref{eq. out ball} we get:
\begin{align}
\nonumber I_{1,j}(x)=&O\left(\int_{B_j}\varepsilon^\alpha\frac1{\left(1+\left|y-\xi'_j\right|\right)^{4-\alpha}}dy\frac1{\left(1+\left|x-\xi'_j\right|\right)^{4-\alpha}}+\int_{B_j}\varepsilon^4dy\sum_{k\ne j}\frac1{\left(1+\left|x-\xi'_k\right|\right)^{4-\alpha}}\right)\\
\label{knotj1}=&O\left(\varepsilon^4+\sum_{k\ne j}^m\frac{\varepsilon^2}{\left(1+\left|x-\xi'_k\right|\right)^{4-\alpha}}\right),
\end{align}
and
\begin{align}
\nonumber I_{2,j}(x)=&O\left(\int_{B_j}\frac1{|x-y|^\alpha}\sum_{j=1}^m\frac1{\left(1+\left|y-\xi'_j\right|\right)^{4-\alpha}}dy\sum_{k\ne j}\frac1{\left(1+\left|x-\xi'_k\right|\right)^{4-\alpha}}\right)\\
\nonumber=&O\left(\varepsilon^\alpha\int_{B_j}\frac1{\left(1+\left|y-\xi'_j\right|\right)^{4-\alpha}}dy\sum_{k\ne j}\frac1{\left(1+\left|x-\xi'_k\right|\right)^{4-\alpha}}\right)\\
\label{knotj2}=&O\left(\sum_{k\ne j}\frac{\varepsilon^\alpha}{\left(1+\left|x-\xi'_k\right|\right)^{4-\alpha}}\right)
\end{align}
On the other hand, if $x\in B_0$, then from \eqref{eq. out ball} we have
\begin{align*}
I_{1,j}(x)=&O\left(\sum_{k=1}^m\int_{B_j}\frac1{|x-y|^\alpha}\frac1{\left(1+\left|y-\xi'_k\right|\right)^{4-\alpha}}dy\frac1{\left(1+\left|x-\xi'_k\right|\right)^{4-\alpha}}\right)\\
=&O\left(\int_{|y-x|<\frac{1+\left|x-\xi'_j\right|}2}\frac1{|x-y|^\alpha}\left(\frac2{1+\left|x-\xi'_j\right|}\right)^{4-\alpha}dy\frac1{\left(1+\left|x-\xi'_k\right|\right)^{4-\alpha}}\right.\\
\nonumber&\left.+\int_{|y-x|\ge\frac{1+\left|x-\xi'_j\right|}2}\left(\frac2{1+\left|x-\xi'_j\right|}\right)^\alpha\frac1{\left(1+\left|y-\xi'_k\right|\right)^{4-\alpha}}dy\frac1{\left(1+\left|x-\xi'_k\right|\right)^{4-\alpha}}\right)\\
=&O\left(\frac1{\left(1+\left|x-\xi'_k\right|\right)^4}\right)\\
=&O\left(\varepsilon^4\right),
\end{align*}
and, similarly,
\begin{align*}
I_{2,j}(x)=&O\left(\int_{B_j}\frac1{|x-y|^\alpha}\int_{B_j}\frac1{|x-y|^\alpha}\frac1{\left(1+\left|y-\xi'_k\right|\right)^{4-\alpha}}dy\sum_{l=1}^m\frac1{\left(1+\left|x-\xi'_l\right|\right)^{4-\alpha}}\right)\\
=&O\left(\int_{B_j}\frac1{|x-y|^\alpha}\int_{B_j}\frac1{|x-y|^\alpha}\frac1{\left(1+\left|y-\xi'_k\right|\right)^{4-\alpha}}dy\frac1{\left(1+\left|x-\xi'_k\right|\right)^{4-\alpha}}\right)\\
=&O\left(\varepsilon^4\right);
\end{align*}
therefore,
$$I_{1,j}(x)-I_{2,j}(x)=O\left(\sum_{k=1}^m\frac{\varepsilon^\alpha}{\left(1+\left|x-\xi'_k\right|\right)^{4-\alpha}}\right).$$
Finally, again from \eqref{eq. out ball} we get
\begin{align*}
I_{1,0}(x)-I_{2,0}(x)=&O\left(\varepsilon^{4-\alpha}\int_{B_0}\frac1{|x-y|^\alpha}dy\sum_{j=1}^m\frac1{\left(1+\left|x-\xi'_j\right|\right)^{4-\alpha}}\right)\\
=&O\left(\varepsilon^{4-\alpha}\sum_{j=1}^m\frac{|x|^{2-\alpha}}{\left(1+\left|x-\xi'_j\right|\right)^{4-\alpha}}\right)\\
=&O\left(\sum_{j=1}^m\frac{\varepsilon^2}{\left(1+\left|x-\xi'_j\right|\right)^{4-\alpha}}\right).
\end{align*}
To sum up,$$
|\mathcal E(x)|=O\left(\sum_{j=1}^m\frac{\varepsilon^{\min\{\alpha,1\}}}{\left(1+\left|x-\xi'_j\right|\right)^{\min\{3,4-\alpha\}}}\right),$$
that is, from the definition \eqref{norm}, $\|\mathcal E\|_\ast=O\left(\varepsilon^{\min\{\alpha,1\}}\right)$.
\end{proof}\

\section{The linearized operator}\

This section is devoted to the study of the linearized operator $\mathcal L\phi$.\\

The starting point of this study is the \emph{limiting linearized operator} around the bubble $w_{1,0}$, defined as:
$$\mathcal L_0\phi:=\Delta\psi(x)+\left(\int_{\mathbb R^2}\frac{e^{w_{1,0}(y)}\phi(y)}{|x-y|^\alpha}dy\right)e^{w_{1,0}(x)}+\left(\int_{\mathbb R^2}\frac{e^{w_{1,0}(y)}dy}{|x-y|^\alpha}\right)e^{w_{1,0}(x)}\phi(x)=0.$$
Bounded solutions to $\mathcal L_0\phi=0$ in $\mathbb R^2$ have been recently classified in \cite{gaol}, where the linearized operator is shown to be non-degenerate. In other words, solutions are linear combinations of the three generators $w_0,w_1,w_2$ defined by:
\begin{align*}
w_0(x):=&\partial_\mu w_{\mu,\xi}(x)|_{\mu=1,\xi=0}=\frac{4-\alpha}2\frac{1-|x|^2}{1+|x|^2};\\
w_i(x):=&\left.\partial_{\xi_i}w_{\mu,\xi}(x)\right|_{\mu=1,\xi=0}=(4-\alpha)\frac{x_i}{1+|x|^2},\quad i=1,2.
\end{align*}
Therefore, since we are working with dilated and translated versions of the standard bubbles, we will consider
\begin{equation}\label{zij}
z_{i,j}(x):=w_i\left(\frac{x-\xi'_j}{\mu_j}\right),\qquad i=0,1,2,j=1,\dots,m.
\end{equation}\

However, it does not seem convenient to consider the full kernel generated by the elements $z_{i,j}$.\\
In fact, imposing orthogonality with respect to all of them is a $3m$-dimensional constraint; this seems not possible to obtain, since the scaling parameters $\mu_j$ are fixed by \eqref{eq. definition of mu} hence the problem has only $2m$ free parameters given by the centers of the bubbles $\xi_1,\dots,\xi_m\in\Omega\subset\mathbb R^2$. For this reason, we will rather discard the elements $z_{0,j}$, related to invariance by dilation, and just consider orthogonality with respect with $z_{i,j}$ with $i=1,2$.\\
In this case, we will not get uniform estimates independently on $\varepsilon$, unless when we include also $i=0$, but the constant will grow logarithmically in $\varepsilon$. However, this will not cause any issues, since such estimates need to be combined with the ones from the previous section, as well as with others from the following sections, which decay polynomially in $\varepsilon$.\\

We will also need some cut-off functions centered at the points $\xi'_j$, to prevent issues for interactions between the elements $z_{i,j},z_{k,l}$ with $j\ne l$.\\
Let $R_0>0$ be large, to be fixed later, and $\chi:[0,+\infty)\to[0,+\infty)$ a smooth non-negative function such that 
$$\left\{\begin{array}{ll}\chi(t)\equiv1&\mbox{if }t<R_0,\\\chi(t)\equiv0&\mbox{if }t>R_0+1,\\0\le\chi\le1\end{array}.\right.$$
Set
$$\chi_j(x)=\chi\left(\left|x-\xi'_j\right|\right),\qquad j=1,\dots,m.$$\

The main result of this section is the following.

\begin{proposition}\label{prop1}
There exist $C>0$ such that, for all $\varepsilon>0$ the problem
\begin{equation}\label{eq. unique solu}
\left\{\begin{array}{ll}
\mathcal L\phi=h+\sum_{i=1}^2\sum_{j=1}^mc_{i,j}\chi_jz_{i,j}&\text{in }\Omega_\varepsilon,\\
\phi=0,&\text{on }\partial\Omega_\varepsilon,\\
\int_{\Omega_\varepsilon}\chi_jz_{i,j}\phi=0,&i=1,2,\,j=1,\dots,m.
\end{array}\right.
\end{equation}
admits a unique solution, satisfying additionally
\begin{equation}\label{eq. prop1 7}
\|\phi\|_\infty\le C\log\left(\frac1\varepsilon\right)\|h\|_\ast.
\end{equation}
\end{proposition}\

In order to prove Proposition \ref{prop1}, we need some auxiliary lemmas.\\
We start with the existence of a barrier function for the operator $\mathcal L$ outside some fixed disks centered at the concentration points.

\begin{lemma}\label{le.maximum}
There exist $R_1>0$ and $Z\in C^2\left(\Omega_\varepsilon\right)$ satisfying
$$\left\{\begin{array}{ll}\mathcal LZ(x)<0&\forall x\in\widetilde\Omega_\varepsilon,\\Z(x)>\frac12&\forall x\in\Omega_\varepsilon,\end{array}\right.$$
where
$$\widetilde\Omega_\varepsilon:=\left\{x\in\Omega_\varepsilon:\,\left|x-\xi'_j\right|\ge R_1,\,\forall j=1,\dots,m\right\}.$$
As a consequence, $\mathcal L$ satisfies the maximum principle in $\widetilde\Omega_\varepsilon$, namely
$$\left\{\begin{array}{ll}\mathcal L\phi(x)<0&\forall x\in\widetilde\Omega_\varepsilon,\\\phi(x)>0&\forall x\in\partial\widetilde\Omega_\varepsilon,\end{array}\right.\qquad\Rightarrow\qquad\phi(x)>0\quad\forall x\in\widetilde\Omega_\varepsilon.$$
\end{lemma}

\begin{proof}
We suffice to prove the existence of such a barrier $Z$, since the maximum principle will then follow after standard arguments.\\
Define
$$Z(x)=1-a\sum_{j=1}^m\frac1{1+\left|x-\xi'_j\right|^2},$$
for some $a>0$ to be chosen later.\\
By definition, one has $Z(x)<1$ and, if $a<\frac1{2m}$, $Z(x)>\frac12$ for any $x\in\Omega_\varepsilon$.\\
By direct computations we have
$$\Delta Z=4a\sum_{j=1}^m\frac{1-\left|x-\xi'_j\right|^2}{\left(1+\left|x-\xi'_j\right|^2\right)^3},$$
therefore, if $|x-\xi'_j|>R_1\ge2$ we get
$$\Delta Z<-\frac{12}5a\sum_{j=1}^m\frac1{\left(1+\left|x-\xi'_j\right|^2\right)^2}\le-\frac{48}5a\sum_{j=1}^m\frac1{\left(1+\left|x-\xi'_j\right|\right)^4}.$$
Moreover, from \eqref{eq. out ball} and $0<Z(x)<1$, we get
\begin{align}
\nonumber&\left(\int_{\Omega_\varepsilon}\frac{k(\varepsilon y)e^{V(y)}}{|x-y|^\alpha}dy\right)k(\varepsilon x)e^{V(x)}Z(x)+\left(\int_{\Omega_\varepsilon}\frac{k(\varepsilon y)e^{V(y)}Z(y)}{|x-y|^\alpha}dy\right)k(\varepsilon x)e^{V(x)}\\
\nonumber\le&C\sum_{j=1}^m\int_{\left\{\left|y-\xi'_j\right|\le\frac C\varepsilon\right\}}\frac1{|x-y|^\alpha}\frac1{\left(1+\left|y-\xi'_j\right|\right)^{4-\alpha}}dy\sum_{k=1}^m\frac1{\left(1+\left|x-\xi'_k\right|\right)^{4-\alpha}}\\
\nonumber\le&C\sum_{j=1}^m\int_{\left\{\left|y-x\right|\le\frac{1+\left|x-\xi'_j\right|}2\right\}}\frac1{|x-y|^\alpha}\left(\frac2{1+\left|x-\xi'_j\right|}\right)^{4-\alpha}dy\sum_{k=1}^m\frac1{\left(1+\left|x-\xi'_k\right|\right)^{4-\alpha}}\\
\nonumber&+C\sum_{j=1}^m\int_{\left\{\left|y-\xi'_j\right|\le\frac C\varepsilon\right\}}\left(\frac2{1+\left|x-\xi'_j\right|}\right)^\alpha\frac1{\left(1+\left|y-\xi'_j\right|\right)^{4-\alpha}}dy\sum_{k=1}^m\frac1{\left(1+\left|x-\xi'_k\right|\right)^{4-\alpha}}\\
\nonumber\le&C\sum_{j=1}^m\frac1{\left(1+\left|x-\xi'_j\right|\right)^2}\sum_{k=1}^m\frac1{\left(1+\left|x-\xi'_k\right|\right)^{4-\alpha}}+C\sum_{j=1}^m\frac1{\left(1+\left|x-\xi'_j\right|\right)^\alpha}\sum_{k=1}^m\frac1{\left(1+\left|x-\xi'_k\right|\right)^{4-\alpha}}\\
\label{integralestimate}\le&C\sum_{j=1}^m\frac1{\left(1+\left|x-\xi'_j\right|\right)^4}.
\end{align}
Hence, if $a$ if sufficiently small and $R_1>2$, we get $\mathcal LZ(x)<0$ for any $x\in\widetilde\Omega_\varepsilon$.
\end{proof}\

The existence of such barrier functions also provides a crucial estimate: roughly speaking, we have uniform a-priori estimates far away from the points $\xi'_j$, hence we suffice to focus on some fixed balls centered at those points.

\begin{lemma}\label{innernorm}
There exists $C>0$ such that, if $\phi,h$ solve
$$\left\{\begin{array}{ll}
\nonumber\mathcal L\phi=h&\text{in }\Omega_\varepsilon,\\
\phi=0,&\text{on }\partial\Omega_\varepsilon,
\end{array}\right.$$
then
$$\|\phi\|_\infty\le C(\|\phi\|_\imath+\|h\|_\ast),$$
where
$$\|\phi\|_\imath:=\sup_{\Omega_\varepsilon\setminus\widetilde\Omega_\varepsilon}|\phi|.$$
\end{lemma}

\begin{proof}
Consider the solution to the problem
$$\left\{\begin{array}{ll}
-\Delta\psi_j(x)=\frac1{\left|x-\xi'_j\right|^{\beta+2}}+\varepsilon^2,&R_1<\left|x-\xi'_j\right|<\frac D\varepsilon,\\
\psi_j(x)=0,&\left|x-\xi'_j\right|=R_1,\frac D\varepsilon,
\end{array}\right.$$
where $\beta:=\min\{1,2-\alpha\}<2$ and $D:=2\mathrm{diam}(\Omega)$, so that $\Omega_\varepsilon$ is contained in the ball centered at $\xi'_j$ of radius $\frac D\varepsilon$. We have the explicit expression
\begin{align*}\psi_j(x)=&\frac1{\beta^2}\left(\frac1{R_1^\beta}-\frac1{\left|x-\xi'_j\right|^\beta}\right)\\
&+\frac{\varepsilon^2}4\left(R_1^2-\left|x-\xi'_j\right|^2\right)+\left(\frac1{\beta^2}\left(\left(\frac\varepsilon D\right)^\beta-\frac1{R_1^\beta}\right)+\frac{\varepsilon^2}4\left(\left(\frac D\varepsilon\right)^2-R_1^2\right)\right)\frac{\log\frac{\left|x-\xi'_j\right|}{R_1}}{\log\frac D{R_1\varepsilon}},
\end{align*}
which gets $\psi_j(x)\le\frac{\varepsilon^2R_1^2+D^2}4$ for all $x\in\widetilde\Omega_\varepsilon$, uniformly bounded independently on $\varepsilon$; moreover, $\psi_j(x)>0$ for all $x\in\widetilde\Omega_\varepsilon$ by the maximum principle.\\
Given $\phi$, we set
$$\phi_0(x):=2\|\phi\|_\imath Z(x)+2\|h\|_\ast\sum_{j=1}^m\psi_j(x),$$
where $Z(x)$ is as in Lemma \ref{le.maximum}.
Since $\psi_j>0$ and $Z<\frac12$, for all $x\in\widetilde\Omega_\varepsilon$ we get
$$\phi_0(x)>2\|\phi\|_\imath Z(x)\ge\|\phi\|_\imath\ge\phi(x).$$
Moreover, since $\psi_j$'s are bounded, arguing as in \eqref{integralestimate} we find
\begin{align*}
&\left(\int_{\Omega_\varepsilon}\frac{k(\varepsilon y)e^{V(y)}}{|x-y|^\alpha}dy\right)k(\varepsilon x)e^{V(x)}\psi_j(x)+\left(\int_{\Omega_\varepsilon}\frac{k(\varepsilon y)e^{V(y)}\psi_j(y)}{|x-y|^\alpha}dy\right)k(\varepsilon x)e^{V(x)}\\
\le&C\int_{\Omega_\varepsilon}\frac1{|x-y|^\alpha}\sum_{j=1}^m\frac1{\left(1+\left|y-\xi'_j\right|\right)^{4-\alpha}}dy
\sum_{k=1}^m\frac1{\left(1+\left|x-\xi'_k\right|\right)^{4-\alpha}}\\
\le&C\sum_{j=1}^m\frac1{\left(1+\left|x-\xi'_j\right|\right)^4}.
\end{align*}
From the definition of $\|\cdot\|_\ast$ we obtain:
\begin{align*}
\mathcal L\phi_0(x)<&2\|h\|_\ast\sum_{j=1}^m\mathcal L\psi_j(x)\\
\le&2\|h\|_\ast\left(-\sum_{j=1}^m\frac1{\left|x-\xi'_j\right|^{\beta+2}}-m\varepsilon^2+\frac C{\left|x-\xi'_j\right|^4}\right)\\
\le&2\|h\|_\ast\left(\sum_{j=1}^m\frac1{\left|x-\xi'_j\right|^{\beta+2}}+\varepsilon^2\right)\left(-1+\frac C{R_1^{2-\beta}}\right)\\
\le&-\|h\|_\ast\left(\sum_{j=1}^m\frac1{\left|x-\xi'_j\right|^{\beta+2}}+\varepsilon^2\right)\\
\le&h(x)\\
=&\mathcal L\phi(x),
\end{align*}
after choosing, if necessary, a larger $R_1$.\\
Therefore, from Lemma \ref{le.maximum} we get $\phi\le\phi_0$ on $\widetilde\Omega_\varepsilon$. Similarly, we get $-\phi\le\phi_0$ in $\widetilde\Omega_\varepsilon$; therefore
$$\|\phi\|_\infty\le\left\|\phi_0\right\|_\infty\le2\|\phi\|_\imath\|Z\|_\infty+2\|h\|_\ast\sum_{j=1}^m\|\phi_j\|_\infty\le C(\|\phi\|_\imath+\|h\|_\ast).$$
\end{proof}\

\begin{lemma}\label{le. inver1}
There exists $C>0$ such that any solution $\phi$ to
\begin{equation}\label{eq. inver 1}
\left\{\begin{array}{ll}
\mathcal L\phi=h&\text{in }\Omega_\varepsilon,\\
\phi=0,&\text{on }\partial\Omega_\varepsilon,\\
\int_{\Omega_\varepsilon}\chi_jz_{i,j}\phi=0,&i=0,1,2,\,j=1,\dots,m.
\end{array}\right.
\end{equation}
satisfies
$$\|\phi\|_\infty\le C\|h\|_\ast.$$
\end{lemma}

\begin{proof}
We argue by contradiction: we assume there exist sequences $\varepsilon_n\to0$, ${\xi'_j}^n\in\Omega_{\varepsilon_n}$, $\phi_n,h_n$ such that $\|h_n\|_\ast\to0,\|\phi_n\|_\infty=1$ and
$$\left\{\begin{array}{ll}
\mathcal L\phi_n=h_n,&\text{in }\Omega_{\varepsilon_n},\\
\phi_n=0,&\text{on }\partial\Omega_{\varepsilon_n},\\
\int_{\Omega_{\varepsilon_n}}\chi_jz_{i,j}\phi_n=0, &i=0,1,2,\,j=1,\dots,m.
\end{array}\right.$$
From Lemma \ref{innernorm}, there exists $C>0$, independent of $n$, such that $\|\phi_n\|_\imath\ge\frac1C$. therefore, for each $n$ there exists an index $j\in\{1,\dots,m\}$ satisfying
\begin{equation}\label{notvanishing}
\sup_{\left|x-{\xi'_j}^n\right|\le R_1}|\phi_n(x)|\ge\frac1C;
\end{equation}
without loss of generality, the same index $j$ works for all $n$.\\
We then re-center the functions around this points by defining
$$\widehat\phi_n(x):=\phi_n\left({\xi'_j}^n+x\right).$$
Standard elliptic estimates imply that, after passing to a further subsequence, $\widehat\phi_n(x)$ converges to a bounded solution $\widehat\phi(x)$ of the problem
$$\Delta\widehat\phi(x)+\left(\int_{\mathbb R^2}\frac{e^{w_{\mu_j,0}(y)}\widehat\phi(y)}{|x-y|^\alpha}dy\right)e^{w_{\mu_j,0}(x)}+\left(\int_{\mathbb R^2}\frac{e^{w_{\mu_j,0}(y)}dy}{|x-y|^\alpha}\right)e^{w_{\mu_j,0}(x)}\widehat\phi(x)=0.$$
From \cite[Theorem 1.1]{gaol}, $\widehat\phi$ must be a linear combination of the functions $z_{i,j}$ for $i=0,1,2$.\\
We may pass to the limit in the orthogonality conditions and get
$$\int_{\mathbb R^2}\chi_j(|x|)z_{i,j}(x)\widehat\phi(x)dx=0,\qquad i=0,1,2,$$
implying $\widehat\phi\equiv0$.\\
On the other hand, \eqref{notvanishing} implies
$$\sup_{|x|\le R_1}\left|\widehat\phi(x)\right|\ge\frac1C>0,$$
which gives the contradiction we wished.
\end{proof}\

We now derive a priori estimates for the linearized problem after removing the orthogonality conditions associated with the dilation modes, namely we consider
\begin{equation}\label{eq. inver 2}
\left\{\begin{array}{ll}
\mathcal L\phi=h&\text{in }\Omega_\varepsilon,\\
\phi=0,&\text{on }\partial\Omega_\varepsilon,\\
\int_{\Omega_\varepsilon}\chi_jz_{i,j}\phi=0,&i=1,2,\,j=1,\dots,m.
\end{array}\right.
\end{equation}\

Before that, we need a technical estimates on some modified version of the kernel generator $z_{0,j}$.

\begin{lemma}\label{ltildez0j}
Consider, for a fixed $j=1,\dots,m$, two smooth positive radial cut-off functions $\eta_{1,j}\left(\left|x-\xi'_j\right|\right)$, $\eta_{2,j}\left(\left|x-\xi'_j\right|\right)$ such that
\begin{equation}\label{eta12}
\left\{\begin{array}{ll}\eta_{1,j}(t)\equiv1&\mbox{if }t<R,\\\eta_{1,j}(t)\equiv0&\mbox{if }t>R+1,\\-C\le\eta_{1,j}'\le0,\\\left|\eta_{1,j}'\right|\le C,\end{array}\right.\qquad\qquad\qquad\left\{\begin{array}{ll}\eta_{2,j}(t)\equiv1&\mbox{if }t<\frac\tau{4\varepsilon},\\\eta_{2,j}(t)\equiv0&\mbox{if }t>\frac\tau{3\varepsilon},\\-C\varepsilon\le\eta_{2,j}'\le0\\\left|\eta_{2,j}''\right|\le C\varepsilon^2,\end{array}\right.
\end{equation}
where $R>R_0+1$ is fixed, and define
$$\widetilde z_{0,j}:=\eta_{1,j}z_{0,j}+(1-\eta_{1,j})\eta_{2,j}h_j z_{0,j},\qquad\qquad\qquad h_j(x):=\frac{\log\frac\tau{\varepsilon\left|x-\xi'_j\right|}}{\log\frac\tau{\varepsilon R}}.$$
Then,
$$\|\mathcal L\widetilde z_{0,j}\|_\ast\le\frac C{\log\frac1\varepsilon}$$
\end{lemma}

\begin{proof}
We first notice that $h_j$ is the solution to the Dirichlet problem
$$\left\{\begin{array}{ll}-\Delta h_j(x)=0&R<\left|x-\xi'_j\right|<\frac\tau\varepsilon,\\h_j(x)=1&\left|x-\xi'_j\right|=R,\\h_j=0&\left|x-\xi'_j\right|=\frac\tau\varepsilon,\end{array}\right.$$
hence in particular it is harmonic.\\
In order to prove the Lemma, we consider the \emph{limiting operator} around the $j$-th bubble as
\begin{equation}
\label{lj}\mathcal L_j\psi:=\Delta\psi+\left(\int_{\mathbb R^2}\frac{e^{w_j(y)}\psi(y)}{|x-y|^\alpha}dy\right)e^{w_j(x)}+\left(\int_{\mathbb R^2}\frac{e^{w_j(y)}}{|x-y|^\alpha}dy\right)e^{w_j(x)}\psi(x).
\end{equation}
Clearly $z_{0,j}$ belongs to the kernel of the limiting operator, namely
\begin{equation}\label{eq. inve2 1}
\mathcal L_jz_{0,j}=0\qquad\text{in }\mathbb R^2.
\end{equation}
From this, we deduce a global estimate for $\mathcal Lz_{0,j}$ on $\Omega_\varepsilon$; in fact, using estimates \eqref{eq. in ball} and \eqref{eq. out ball} and the boundedness of $\|z_{0,j}\|_\infty$, arguing as in Proposition \ref{le. estimate of E} one gets:
\begin{align}
\nonumber|\mathcal Lz_{0,j}|=&\left|\left(\int_{\Omega_\varepsilon}\frac{k(\varepsilon y)e^{V(y)}z_{0,j}(y)}{|x-y|^\alpha}dy\right)k(\varepsilon x)e^{V(x)}-\left(\int_{\mathbb R^2}\frac{e^{w_j(y)}z_{0,j}(y)}{|x-y|^\alpha}dy\right)e^{w_j(x)}\right.\\
\nonumber&\left.+\left(\int_{\Omega_\varepsilon}\frac{k(\varepsilon y)e^{V(y)}}{|x-y|^\alpha}dy\right)k(\varepsilon x)e^{V(x)}z_{0,j}(x)-\left(\int_{\mathbb R^2}\frac{e^{w_j(y)}}{|x-y|^\alpha}dy\right)e^{w_j(x)}z_{0,j}(x)\right|\\
\nonumber\le&C\left(\|z_{0,j}\|_\infty\int_{B_j}\frac{\varepsilon\left(\left|y-\xi'_j\right|+\left|x-\xi'_j\right|\right)}{|x-y|^\alpha}\frac1{\left(1+\left|y-\xi'_j\right|\right)^{4-\alpha}}dy\frac1{\left(1+\left|x-\xi'_j\right|\right)^{4-\alpha}}\right.\\
\nonumber&+\left.\|z_{0,j}\|_\infty\sum_{k\ne j}\int_{B_k}\frac1{|x-y|^\alpha}\frac1{(1+|y-\xi'_i|)^{4-\alpha}}dy\frac1{\left(1+\left|x-\xi'_j\right|\right)^{4-\alpha}}\right.\\
\nonumber&\left.+\varepsilon^{4-\alpha}\|z_{0,j}\|_\infty\int_{B_0}\frac1{|x-y|^\alpha}dy\frac1{\left(1+\left|x-\xi'_j\right|\right)^{4-\alpha}}\right.\\
\nonumber&\left.+\|z_{0,j}\|_\infty\int_{\mathbb R^2\setminus\Omega_\varepsilon}\frac1{|x-y|^\alpha}\frac1{\left(1+\left|y-\xi'_j\right|\right)^{4-\alpha}}dy\frac1{\left(1+\left|x-\xi'_j\right|\right)^{4-\alpha}}\right)\\
\nonumber\le&C\left(\|z_{0,j}\|_\infty\left(\frac{\varepsilon^{\min\{\alpha,1\}}}{\left(1+\left|x-\xi'_j\right|\right)^{\min\{3,4-\alpha\}}}+\frac{\varepsilon^2}{\left(1+\left|x-\xi'_j\right|\right)^{4-\alpha}}\right)\right)\\
\label{lz0j}\le&\frac{C\varepsilon^{\min\{\alpha,1\}}}{\left(1+\left|x-\xi'_j\right|\right)^{\min\{3,4-\alpha\}}}.
\end{align}
We now split $\Omega_\varepsilon$ in four sub-regions, according to the definition of $\widetilde z_{0,j}$:
\begin{align}
\label{omegai,j}\Omega_{1,j}:=&\left\{x\in\Omega_\varepsilon:\,\left|x-\xi_j\right|<R\right\},\\
\nonumber\Omega_{2,j}:=&\left\{x\in\Omega_\varepsilon:\,R\le\left|x-\xi_j\right|<R+1\right\},\\
\nonumber\Omega_{3,j}:=&\left\{x\in\Omega_\varepsilon:\,R+1\le\left|x-\xi_j\right|<\frac\tau{4\varepsilon}\right\},\\
\nonumber\Omega_{4,j}:=&\left\{x\in\Omega_\varepsilon:\,\left|x-\xi_j\right|\ge\frac\tau{4\varepsilon}\right\};
\end{align}
then, we estimate $\mathcal L\widetilde z_{0,j}$ on each region.\\
\begin{itemize}
\item In $\Omega_{1,j}$ we have $\widetilde z_{0,j}=z_{0,j}$, hence
\begin{equation}\label{eq. inve2 4}|\mathcal L\widetilde z_{0,j}(x)|=|\mathcal Lz_{0,j}(x)|\le\frac{C\varepsilon^{\min\{\alpha,1\}}}{\left(1+\left|x-\xi'_j\right|\right)^{\min\{3,4-\alpha\}}}.
\end{equation}
\item In $\Omega_{2,j}$, we have
$$\widetilde z_{0,j}=z_{0,j}-(1-\eta_{1,j})(1-h_j)z_{0,j}.$$
In this region $z_{0,j}$ and $\eta_{1,j}$ are uniformly bounded in $C^2$; moreover,
$$|1-h_j|+|\nabla h_j|\le\frac C{\log\frac1\varepsilon}.$$
Therefore, since $h_j$ is harmonic,
\begin{equation}\label{diffz0j}
|\widetilde z_{0,j}-z_{0,j}|+|\Delta(\widetilde z_{0,j}-z_{0,j})|\le\frac C{\log\frac1\varepsilon},
\end{equation}
hence
\begin{align}
\nonumber|\mathcal L\widetilde z_{0,j}(x)|\le&|\mathcal L(\widetilde z_{0,j}-z_{0,j})(x)|+|\mathcal Lz_{0,j}(x)|\\
\nonumber\le&\frac C{\log\frac1\varepsilon}+\frac{C\varepsilon^{\min\{\alpha,1\}}}{\left(1+\left|x-\xi'_j\right|\right)^{\min\{3,4-\alpha\}}}\\
\label{ltildez0j2}\le&\frac C{\log\left(\frac1\varepsilon\right)\left(1+\left|x-\xi'_j\right|\right)^{\min\{3,4-\alpha\}}}.
\end{align}
\item In $\Omega_{3,j}$ we have $\widetilde z_{0,j}=h_jz_{0,j}$; therefore, using again \eqref{eq. inve2 1} and then again the harmonicity of $h_j$, we get
\begin{align*}
&|\mathcal L\widetilde z_{0,j}(x)|\\
\le&|\mathcal L_j(h_jz_{0,j})(x)-h_j(x)\mathcal L_jz_{0,j}(x)|+|h_j(x)(\mathcal L-\mathcal L_j)z_{0,j}(x)|\\
\le&\left|2\nabla h_j(x)\cdot\nabla z_{0,j}(x)+\left(\int_{\mathbb R^2}\frac{e^{w_j(y)}z_{0,j}(y)h_j(y)}{|x-y|^\alpha}dy\right)e^{w_j(x)}-\left(\int_{\mathbb R^2}\frac{e^{w_j(y)}z_{0,j}(y)}{|x-y|^\alpha}dy\right)e^{w_j(x)}h_j(x)\right|\\
&+\|h_j\|_\infty|(\mathcal L-\mathcal L_j)z_{0,j}(x)|\\
\le&2|\nabla h_j(x)||\nabla z_{0,j}(x)|+\frac C{\log\frac1\varepsilon}\underbrace{\int_{\mathbb R^2}\frac{\left|\log\frac{\left|x-\xi'_j\right|}{\left|y-\xi'_j\right|}\right|}{|x-y|^\alpha}\frac1{\left(1+\left|y-\xi'_j\right|\right)^{4-\alpha}}dy\frac1{\left(1+\left|x-\xi'_j\right|\right)^{4-\alpha}}}_{=:I(x)}\\
&+\frac{C\varepsilon^{\min\{\alpha,1\}}}{\left(1+\left|x-\xi'_j\right|\right)^{\min\{3,4-\alpha\}}}
\end{align*}
From the explicit expressions of $h_j,z_{0,j}$ we get
$$|\nabla h_j(x)||\nabla z_{0,j}(x)|\le\frac C{\left(1+\left|x-\xi'_j\right|\right)^4\log\frac1\varepsilon}.$$
We therefore suffice to estimate $I(x)$:
\begin{align*}
|I(x)|\le&C\left(\int_{\left\{|y-x|\le\frac{\left|x-\xi'_j\right|}2\right\}}|x-y|^{1-\alpha}\frac1{\left(1+\frac{\left|x-\xi'_j\right|}2\right)^{4-\alpha}}dy\frac1{\left(1+\left|x-\xi'_j\right|\right)^{4-\alpha}}\right.\\
&\left.+\int_{\left\{|y-x|>\frac{\left|x-\xi'_j\right|}2\right\}}\frac{\left|\log\left|x-\xi'_j\right|\right|+\left|\log\left|y-\xi'_j\right|\right|}{\left(\frac{\left|x-\xi'_j\right|}2\right)^\alpha}\frac1{\left(1+\left|y-\xi'_j\right|\right)^{4-\alpha}}dy\frac1{\left(1+\left|x-\xi'_j\right|\right)^{4-\alpha}}\right)\\
\le&C\left(\frac1{\left(1+\left|x-\xi'_j\right|\right)^{6-\alpha}}+\frac{1+\left|\log\left|x-\xi'_j\right|\right|}{\left(1+\left|x-\xi'_j\right|\right)^4}\right)\\
\le&\frac C{\left(1+\left|x-\xi'_j\right|\right)^{4-\alpha}};
\end{align*}
this shows that, again,
\begin{equation}\label{ltildez0j3}
|\mathcal L\widetilde z_{0,j}|\le\frac C{\log\left(\frac1\varepsilon\right)\left(1+\left|x-\xi'_j\right|\right)^{\min\{3,4-\alpha\}}}
\end{equation}
\item In $\Omega_{4,j}$ we have $\widetilde z_{0,j}=\eta_{2,j}h_jz_{0,j}$. Moreover,
$$|h_j(x)|\le\frac C{\log\frac1\varepsilon},\qquad\qquad\qquad|\nabla h_j(x)|\le\frac C{\log\left(\frac1\varepsilon\right)\left(1+\left|x-\xi'_j\right|\right)};$$
therefore, using the properties \eqref{eta12} of $\eta_{2,j}$ and the explicit expression of $z_{0,j}$ we get:
$$|\Delta\widetilde z_{0,j}(x)|=|\Delta(\eta_{2,j}(x)h_j(x)z_{0,j}(x))|\le\frac C{\log\frac1\varepsilon}\left(\frac1{\left(1+\left|x-\xi'_j\right|\right)^4}+\frac\varepsilon{\left(1+\left|x-\xi'_j\right|\right)^2}+\varepsilon^2\right)\le\frac{C\varepsilon^2}{\log\frac1\varepsilon}$$
Since here $\left|x-\xi'_j\right|>\frac\tau{4\varepsilon}$, then an estimate similar to \eqref{eq. out ball} holds true; therefore, arguing as in \eqref{integralestimate} one gets:
\begin{align*}
&\left|\left(\int_{\Omega_\varepsilon}\frac{k(\varepsilon y)e^{V(y)}}{|x-y|^\alpha}dy\right)k(\varepsilon x)e^{V(x)}\widetilde z_{0,j}(x)+\left(\int_{\Omega_\varepsilon}\frac{k(\varepsilon y)e^{V(y)}\widetilde z_{0,j}(y)}{|x-y|^\alpha}dy\right)k(\varepsilon x)e^{V(x)}\right|\\
\le&C\varepsilon^{4-\alpha}\left(\sum_{j=1}^m\int_{B_j}\frac1{|x-y|^\alpha}\frac1{\left(1+\left|y-\xi'_j\right|\right)^{4-\alpha}}dy+\varepsilon^{4-\alpha}\int_{B_0}\frac1{|x-y|^\alpha}dy\right)\\
\le&C\left(\sum_{j=1}^m\frac{\varepsilon^{4-\alpha}}{\left(1+\left|x-\xi'_j\right|\right)^\alpha}+\varepsilon^{6-\alpha}\right)\\
\le&C\varepsilon^4;
\end{align*}
thus,
\begin{equation}\label{ltildez0j4}
|\mathcal L\widetilde z_{0,j}(x)|\le\frac{C\varepsilon^2}{\log\frac1\varepsilon}.
\end{equation}
\end{itemize}
We conclude by putting together with \eqref{eq. inve2 4}, \eqref{ltildez0j2}, \eqref{ltildez0j3} and \eqref{ltildez0j4}.
\end{proof}

A crucial step in the proof of Proposition \ref{prop1} is the following variation of Lemma \ref{le. inver1}: after removing the orthogonality condition with respect to $z_{0,j}$, the norm of the inverse operator of $\mathcal L$ grows proportionally to the logarithm of $\varepsilon$.

\begin{lemma}\label{le. inver2}
Then there exists $C>0$ such that any solution $\phi$ of \eqref{eq. inver 2} satisfies
$$\|\phi\|_\infty\le C\log\left(\frac1\varepsilon\right)\|h\|_\ast.$$
\end{lemma}

\begin{proof}
Given $\phi$ solving \eqref{eq. inver 2}, we modify it so that it also verifies the extra condition in \eqref{eq. inver 1}. More precisely, we define
$$\widetilde\phi:=\phi+\sum_{j=1}^md_j\widetilde z_{0,j},$$
where $\widetilde z_{0,j}$ is as in Lemma \ref{ltildez0j} and
\begin{equation}\label{eq. inve2 9}
d_j:=-\frac{\int_{\Omega_\varepsilon}\chi_j\widetilde z_{0,j}\phi}{\int_{\Omega_\varepsilon}\chi_j\widetilde z_{0,j}^2}.
\end{equation}
We notice that, if $\left|x-\xi'_j\right|\le R$, then $\widetilde z_{0,j}=z_{0,j}$; therefore, since $R>R_0+1$, $\widetilde z_{0,j}$ equals $z_{0,j}$ in the support of $\chi_j$, hence the choice \eqref{eq. inve2 9} of $d_j$ gives
$$\int_{\Omega_\varepsilon}\chi_jz_{0,j}\widetilde\phi=\int_{\Omega_\varepsilon}\chi_j\widetilde z_{0,j}\widetilde\phi=0\qquad\qquad j=1,\dots,m;$$
on the other hand, since $z_{0,j}$ and $z_{i,k}$ are mutually orthogonal for any $i=1,2$ and $j,k=1,\dots,m$, then we also get
$$\int_{\Omega_\varepsilon}\chi_jz_{i,j}\widetilde\phi=0,\qquad\qquad i=1,2,\quad j=1,\dots,m.$$
Since
\begin{equation}\label{ltildephi}
\mathcal L\widetilde\phi=h+\sum_{j=1}^md_j\mathcal L\widetilde z_{0,j},
\end{equation}
then $\widetilde\phi$ solves \eqref{eq. inver 1} with $h+\sum_{j=1}^md_j\mathcal L\widetilde z_{0,j}$ in place of $h$ and we may apply Lemma \ref{le. inver1} and get
\begin{equation}\label{eq. inver 11}
\left\|\widetilde\phi\right\|_\infty\le C\left(\|h\|_\ast+\sum_{j=1}^m|d_j|\|\mathcal L\widetilde z_{0,j}\|_\ast\right).
\end{equation}
Therefore,
$$\|\phi\|_\infty\le\left\|\widetilde\phi\right\|_\infty+\sum_{j=1}^m|d_j|\|\widetilde z_{0,j}\|_\infty\le C\left(\|h\|_\ast+\sum_{j=1}^m|d_j|\left(1+\left\|\mathcal L\widetilde z_{0,j}\right\|_\ast\right)\right);$$
in view of the estimate for $\left\|\mathcal L\widetilde z_{0,j}\right\|_\ast$ from Lemma \ref{ltildez0j}, we suffice to show that
$$|d_j|\le C\log\left(\frac1\varepsilon\right)\|h\|_\ast.$$
By multiplying equation \eqref{ltildephi} by $\widetilde z_{0,k}$, for a fixed $k\in\{1,\dots,m\}$, and integrate, in these cases we get:
$$\int_{\Omega_\varepsilon}\widetilde\phi\mathcal L\widetilde z_{0,k}=\int_{\Omega_\varepsilon}\widetilde z_{0,k}\mathcal L\widetilde\phi=\int_{\Omega_\varepsilon}h\widetilde z_{0,k}+\sum_{j=1}^md_j\int_{\Omega_\varepsilon}\widetilde z_{0,j}\mathcal L\widetilde z_{0,k}.$$
Since $z_{0,k},z_{0,j}$ have disjoint support for $k\ne j$, then, using \eqref{eq. in ball}, we get
\begin{align*}
\left|\int_{\Omega_\varepsilon}\widetilde z_{0,j}\mathcal L\widetilde z_{0,k}\right|=&2\left|\left(\int_{\Omega_\varepsilon}\frac{k(\varepsilon y)e^{V(y)}}{|x-y|^\alpha}dy\widetilde z_{0,k}(y)\right)dy\int_{\Omega_\varepsilon}k(\varepsilon x)e^{V(x)}\widetilde z_{0,j}(x)dx\right|\\
\le&C\left(\int_{B_i}\frac1{|x-y|^\alpha}\frac1{(1+|y-\xi'_k|)^{4-\alpha}}dy\int_{B_j}\frac1{\left(1+\left|x-\xi'_j\right|\right)^{4-\alpha}}dx\right)\\
\le&C\varepsilon^\alpha.
\end{align*}
Therefore, using \eqref{eq. inver 11} and then Lemma \ref{ltildez0j} we get:
\begin{align*}
-d_j\int_{\Omega_\varepsilon}\widetilde z_{0,j}\mathcal L\widetilde z_{0,j}\le&\left|\int_{\Omega_\varepsilon}\widetilde\phi\mathcal L\widetilde z_{0,j}\right|+\left|\int_{\Omega_\varepsilon}h\widetilde z_{0,j}\right|+C\varepsilon^\alpha\max_k|d_k|\\
\le&\left\|\widetilde\phi\right\|_\infty\|\mathcal L\widetilde z_{0,j}\|_\ast+\|\widetilde z_{0,j}\|_\infty\|h\|_\ast+C\varepsilon^\alpha\max_k|d_k|\\
\le&C\left(\|h\|_\ast\left(1+\max_k\left\|\mathcal L\widetilde z_{0,k}\right\|_\ast\right)+\max_k\left(|d_k|\left\|\mathcal L\widetilde z_{0,k}\right\|_\ast\right)+\varepsilon^\alpha\max_k|d_k|\right)\\
\le&C\left(\|h\|_\ast+\frac{\max_k|d_k|}{\log\frac1\varepsilon}\right).
\end{align*}
To conclude, we need to show
$$\int_{\Omega_\varepsilon}\widetilde z_{0,j}\mathcal L\widetilde z_{0,j}\le-\frac C{\log\frac1\varepsilon}.$$
As in the proof of Lemma \ref{ltildez0j}, we write $\Omega_\varepsilon=\bigcup_{i=1}^4\Omega_{i,j}$, where $\Omega_{i,j}$ are the disjoint sets defined in \eqref{omegai,j}. We will also use some estimates from Lemma \ref{ltildez0j}.\\
\begin{itemize}
\item From \eqref{eq. inve2 4} we get
$$\left|\int_{\Omega_{1,j}}\widetilde z_{0,j}\mathcal L\widetilde z_{0,j}\right|\le C\varepsilon^{\min\{\alpha,1\}}.$$
\item Similarly, if $x\in\Omega_{4,j}$, we have
$$|\widetilde z_{0,j}|=|\eta_{2,j}||h_j||z_{0,j}|\le\frac C{\log\frac1\varepsilon};$$
therefore, from \eqref{ltildez0j4},
$$\left|\int_{\Omega_{4,j}}\widetilde z_{0,j}\mathcal L\widetilde z_{0,j}\right|\le\frac C{\log^2\frac1\varepsilon}.$$
\item In $\Omega_{3,j}$ we use \eqref{ltildez0j3} to get
$$\left|\int_{\Omega_{3,j}}\widetilde z_{0,j}\mathcal L\widetilde z_{0,j}\right|\le\frac C{\log\frac1\varepsilon}\int_{\left\{\left|x-\xi'_j\right|>R\right\}}\frac1{\left|x-\xi'_j\right|^{\min\{3,4-\alpha\}}}dx\le\frac{C_0}{R^{\min\{1,2-\alpha\}}\log\frac1\varepsilon}.$$
Here, $C_0$ is independent on $R$; therefore, we are left with showing that, in the remaining region,
\begin{equation}\label{eq. inve2 8}
\int_{\Omega_{3,j}}\widetilde z_{0,j}\mathcal L\widetilde z_{0,j}\le-\frac{C_1}{\log\frac1\varepsilon},
\end{equation}
with $C_1$ independent on $R$, and then choosing $R>\left(\frac{C_0}{C_1}\right)^\frac1{\min\{1,2-\alpha\}}$, so that $\frac{C_0}{R^{\min\{1,2-\alpha\}}}-C_1<0$.
\item In $\Omega_{2,j}$, from \eqref{lz0j} we get
\begin{equation}\label{tildezlz}
\left|\int_{\Omega_{2,j}}\widetilde z_{0,j}\mathcal Lz_{0,j}\right|\le C\varepsilon^{\min\{\alpha,1\}}.
\end{equation}
On one hand, from \eqref{diffz0j} we get
\begin{equation}\label{diffldiff}
\left|\int_{\Omega_{2,j}}(\widetilde z_{0,j}-z_{0,j})\mathcal L(\widetilde z_{0,j}-z_{0,j})\right|\le\frac C{\log^2\frac1\varepsilon}.
\end{equation}
moreover, arguing similarly as \eqref{integralestimate} we get:
\begin{align}
\nonumber&\left|\left(\int_{\Omega_\varepsilon}\frac{k(\varepsilon y)e^{V(y)}}{|x-y|^\alpha}dy\right)k(\varepsilon x)e^{V(x)}(z_{0,j}(x)-\widetilde z_{0,j}(x))\right.\\
\nonumber&\left.+\left(\int_{\Omega_\varepsilon}\frac{k(\varepsilon y)e^{V(y)}(z_{0,j}(y)-\widetilde z_{0,j}(y))}{|x-y|^\alpha}dy\right)k(\varepsilon x)e^{V(x)}\right|\\
\nonumber\le&\frac C{\log\frac1\varepsilon}\int_{\Omega_\varepsilon}\frac1{|x-y|^\alpha}\frac1{\left(1+\left|y-\xi'_j\right|\right)^{4-\alpha}}dy\frac1{\left(1+\left|x-\xi'_j\right|\right)^{4-\alpha}}\\
\nonumber\le&\frac C{\log\frac1\varepsilon}\frac1{\left(1+\left|x-\xi'_j\right|\right)^4}\\
\label{nonlocdiff}\le&\frac{C'}{R^4\log\frac1\varepsilon},
\end{align}
with $C'$ independent on $R$.\\
On the other hand, using \eqref{eta12} and the explicit expression of $h_j,z_{0,j}$ we get
$$\Delta(\widetilde z_{0,j}-z_{0,j})=2z_{0,j}\nabla\eta_{1,j}\cdot\nabla h_j-(1-h_j)z_{0,j}\Delta\eta_{1,j}+O\left(\frac1{R^3\log\frac1\varepsilon}\right);$$
therefore, integrating by parts we get
\begin{align*}
\int_{\Omega_{2,j}}z_{0,j}\Delta(\widetilde z_{0,j}-z_{0,j})=&\int_{\Omega_{2,j}}\left(2z_{0,j}^2\nabla\eta_{1,j}\cdot\nabla h_j-(1-h_j)z_{0,j}^2\Delta\eta_{1,j}\right)+O\left(\frac1{R^2\log\frac1\varepsilon}\right)\\
=&\int_{\Omega_{2,j}}\left(z_{0,j}^2\nabla\eta_{1,j}\cdot\nabla h_j+(1-h_j)z_{0,j}\nabla z_{0,j}\cdot\nabla\eta_{1,j}\right)+O\left(\frac1{R^2\log\frac1\varepsilon}\right)\\
=&\frac{2\pi}{\log\frac\tau{\varepsilon R}}\left(\frac{4-\alpha}2\right)^2\underbrace{\int_R^{R+1}\left(\frac{1-r^2}{1+r^2}\right)^2\eta_{1,j}'(r)dr}_{=:J}+O\left(\frac1{R^2\log\frac1\varepsilon}\right).
\end{align*}
We have $\log\frac\tau{\varepsilon R}\le\log\tau+\log\frac1\varepsilon$ and $J$ bounded from above by a negative constant, independently on $R>2$, because
$$J\le\left(\frac{1-R^2}{1+R^2}\right)^2\int_R^{R+1}\eta_{1,j}'\le-\frac9{25};$$
this fact, along with \eqref{tildezlz}, \eqref{diffldiff} and \eqref{nonlocdiff}, will get \eqref{eq. inve2 8} and conclude the proof, by choosing a suitably large $R$.
\end{itemize}
\end{proof}\

We are finally in position to prove the main result of this section.
\begin{proof}[Proof of Proposition \ref{prop1}]
We start with proving the a priori estimate.\\
Applying Lemma \ref{le. inver2} to the equation \eqref{eq. unique solu} we get
\begin{equation}\label{eq. prop1 1}
\|\phi\|_\infty\le C\log\left(\frac1\varepsilon\right)\left(\|h\|_\ast+\sum_{i=1}^2\sum_{j=1}^m|c_{i,j}|\right);
\end{equation}
therefore, we suffice to estimate the constants $c_{i,j}$.\\
We fix $k\in\{1,2\},l\in\{1,\dots,m\}$ and multiply equation \eqref{eq. unique solu} by $\eta_{2,l}z_{k,l}$, where $\eta_{2,l}$ is a cut-off function as in \eqref{eta12}:
$$\int_{\Omega_\varepsilon}(\mathcal L\phi)\eta_{2,l}z_{k,l}=\int_{\Omega_\varepsilon}h\eta_{2,l}z_{k,l}+\sum_{i=1}^2\sum_{j=1}^mc_{i,j}\int_{\Omega_\varepsilon}\chi_jz_{i,j}\eta_{2,l}z_{k,l}=\int_{\Omega_\varepsilon}h\eta_{2,l}z_{k,l}+c_{k,l}\int_{\Omega_\varepsilon}\chi_j\eta_{2,j}z_{k,j}^2,$$
since the cut-off functions $\chi_j,\eta_{2,l}$ have disjoint supports for $j\ne l$, and the functions $z_{1,j},z_{2,j}$ are mutually orthogonal. Therefore, since
\begin{equation}\label{integralckl}\int_{\Omega_\varepsilon}\chi_j\eta_{2,j}z_{k,j}^2=\int_{\Omega_\varepsilon}\chi_jz_{k,j}^2=(4-\alpha)^2\int_{|x|\le R}\chi(x)\frac{x_1^2}{\left(1+|x|^2\right)^2}dx
\end{equation}
is a positive constant, then
\begin{equation}\label{eq. prop1 3}
|c_{k,l}|\le C\left(\left|\int_{\Omega_\varepsilon}(\mathcal L\phi)\eta_{2,l}z_{k,l}\right|+\left|\int_{\Omega_\varepsilon}h\eta_{2,l}z_{k,l}\right|\right).
\end{equation}
The latter term can be estimated as
\begin{equation}\label{eq. prop1 4}
\left|\int_{\Omega_\varepsilon}h\eta_{2,l}z_{k,l}\right|\le C\|h\|_\ast\|\eta_{2,l}z_{k,l}\|_\infty\le C\|h\|_\ast.
\end{equation}
For the first term in \eqref{eq. prop1 3} we recall that $\mathcal L_lz_{k,l}=0$, where $\mathcal L_l$ is the limiting operator defined in \eqref{lj}; therefore, using \eqref{eta12} and the explicit expression of $z_{k,l}$ and then arguing as in \eqref{lz0j} we get:
\begin{align*}
\mathcal L(\eta_{2,l}z_{k,l})(x)=&\mathcal L(\eta_{2,l}z_{k,l})(x)-\eta_{2,l}(x)\mathcal L_lz_{k,l}(x)\\
\le&|z_{k,l}(x)||\Delta\eta_{2,l}(x)|+2|\nabla z_{k,l}(x)||\nabla\eta_{2,l}(x)|+|\eta_{2,l}(x)|(\mathcal L-\mathcal L_l)z_{k,l}(x)|\\
\le&C\left(\frac{\varepsilon^2}{1+\left|x-\xi'_l\right|}+\frac\varepsilon{\left(1+\left|x-\xi'_l\right|\right)^2}+\frac{\varepsilon^{\min\{\alpha,1\}}}{\left(1+\left|x-\xi'_l\right|\right)^{\min\{3,4-\alpha\}}}\right),
\end{align*}
hence, since $\mathcal L$ is self-adjoint,
\begin{equation}\label{eq. prop1 5}
\left|\int_{\Omega_\varepsilon}(\mathcal L\phi)\eta_{2,l}z_{k,l}\right|=\left|\int_{\Omega_\varepsilon}\phi\mathcal L(\eta_{2,l}z_{k,l})\right|\le C\varepsilon^{\min\{\alpha,1\}}\|\phi\|_\infty.
\end{equation}
Combining \eqref{eq. prop1 3}, \eqref{eq. prop1 4} and \eqref{eq. prop1 5} and then \eqref{eq. prop1 1} we get
\begin{equation}\label{ckl}
|c_{k,l}|\le C\left(\|h\|_\ast+\varepsilon^{\min\{\alpha,1\}}\|\phi\|_\infty \right)\le C\|h\|_\ast;
\end{equation}
applying again \eqref{eq. prop1 1} gives the a priori estimate \eqref{eq. prop1 7}.\\
Now we prove the existence and uniqueness. Consider the orthogonal complement of the kernel elements
\begin{equation}\label{k}
\mathbf K:=\left\{\phi\in H^1_0(\Omega_\varepsilon):\int_{\Omega_\varepsilon}\chi_jz_{i,j}\phi=0,\quad i=1,2,\,j=1,\dots,m\right\};
\end{equation}
from the definition of $\mathcal L$, for any $\psi\in\mathbf K$ one has
$$\int_{\Omega_\varepsilon}\psi\mathcal L\phi=\int_{\Omega_\varepsilon}h\psi,$$
namely
\begin{align}\label{eq. prop1 8}
\nonumber&\int_{\Omega_\varepsilon}\nabla\phi\cdot\nabla\psi-\left(\int_{\Omega_\varepsilon}\frac{k(\varepsilon y)e^{V(y)}}{|x-y|^\alpha}dy\right)\int_{\Omega_\varepsilon}k(\varepsilon x)e^{V(x)}\phi(x)\psi(x)dx\\
\nonumber&-\left(\int_{\Omega_\varepsilon}\frac{k(\varepsilon y)e^{V(y)}}{|x-y|^\alpha}dy\right)\int_{\Omega_\varepsilon}k(\varepsilon x)e^{V(x)}\phi(x)\psi(x)dx\\
=&-\int_{\Omega_\varepsilon}h\psi.
\end{align}
Therefore, \eqref{eq. unique solu} is equivalent to \eqref{eq. prop1 8} in $\mathbf K$.\\
From HLS inequality and compact Sobolev embeddings, the nonlocal terms in \eqref{eq. prop1 8} define compact operators from $H^1_0(\Omega_\varepsilon)$ to $H^{-1}_0(\Omega_\varepsilon)$; using then Riesz's representation Theorem, we can write \eqref{eq. prop1 8} as an equation of the form $\phi-\mathcal J\phi=\widetilde h$ in $\mathbf K$, for some $\widetilde h\in\mathbf K$, where $J:\mathbf K\to\mathbf K$ is compact.\\
The a priori estimate \eqref{eq. prop1 7} shows in particular that the only solution to $\phi-\mathcal J\phi=0$ is $\phi=0$. Therefore, by Fredhlom's alternative, the operator is invertible on $\mathbf K$, namely for every $h$ there exists a unique solution $\phi$ to \eqref{eq. unique solu}.
\end{proof}\

\section{The nonlinear problem}\

In this section we will solve the problem \eqref{probphi} projected on the orthogonal complement $\mathbf K$ of the kernel defined by \eqref{k}. This is equivalent to the following equation:
\begin{equation}\label{eq. projected}
\left\{\begin{array}{ll}
\mathcal L\phi=\mathcal E-\mathcal N(\phi)+\sum_{i=1}^2\sum_{j=1}^mc_{i,j}\chi_jz_{i,j}&\text{in }\Omega_\varepsilon,\\
\phi=0,&\text{on }\partial\Omega_\varepsilon,\\
\int_{\Omega_\varepsilon}\chi_jz_{i,j}\phi=0,&i=1,2,\,j=1,\dots,m.
\end{array}\right.
\end{equation}
To this purpose, we use the linear theory established by Proposition \ref{prop1} and a contraction mapping argument.\\

We first need some basic estimate concerning the nonlinear term $\mathcal N(\phi)$:

\begin{lemma}\label{le. estimate of N}
There exists $C>0$ such that, for any $\varepsilon>0$ we have
\begin{align*}
\|\mathcal N(\phi)\|_\ast\le&C\|\phi\|^2_\infty e^{2\|\phi\|_\infty},\\
\|\mathcal N(\phi_1)-\mathcal N(\phi_2)\|_\ast\le&C(\|\phi_1\|_\infty+\|\phi_2\|_\infty)e^{2(\|\phi_1\|_\infty+\|\phi_2\|_\infty)}\|\phi_1-\phi_2\|_\infty.
\end{align*}
\end{lemma}

\begin{proof}
From elementary inequalities we get
$$\left|e^{\phi(x)}e^{\phi(y)}-1-\phi(x)-\phi(y)\right|\le\frac{(\phi(x)+\phi(y))^2}2e^{|\phi(x)+\phi(y)|}\le2\|\phi\|_\infty^2e^{2\|\phi\|_\infty},$$
therefore, arguing as in \eqref{integralestimate}, one gets
\begin{align*}
\mathcal N(\phi)(x)=&O\left(\|\phi\|^2_\infty e^{2\|\phi\|_\infty}\sum_{j=1}^m\int_{\Omega_\varepsilon}\frac1{|x-y|^\alpha}\frac1{\left(1+\left|y-\xi'_j\right|\right)^{4-\alpha}}dy\frac1{\left(1+\left|x-\xi'_j\right|\right)^{4-\alpha}}\right)\\
=&O\left(\|\phi\|^2_\infty e^{2\|\phi\|_\infty}\frac1{\left(1+\left|x-\xi'_j\right|\right)^4}\right)
\end{align*}
which proves the first statement.\\
The second statement is proved similarly, using the elementary estimate
\begin{align*}
&\left|e^{\phi_1(x)}e^{\phi_1(y)}-\phi_1(x)-\phi_1(y)-e^{\phi_2(x)}e^{\phi_2(y)}-\phi_2(x)-\phi_2(y)\right|\\
\le&\frac{\left|(\phi_1(x)+\phi_1(y))^2-(\phi_2(x)+\phi_2(y))^2\right|}2e^{|\phi_1(x)+\phi_1(y)|+|\phi_2(x)+\phi_2(y)|}\\
\le&2(\|\phi_1\|_\infty+\|\phi_2\|_\infty)e^{2(\|\phi_1\|_\infty+\|\phi_2\|_\infty)}\|\phi_1-\phi_2\|_\infty.
\end{align*}
\end{proof}\

The main result of this section is the following.

\begin{proposition}\label{le. unique solu}
There exists $M>0$ such that \eqref{eq. projected} has a unique solution $\phi$ satisfying
$$\|\phi\|_\infty\le M\varepsilon^{\min\{\alpha,1\}}\log\frac1\varepsilon.$$
\end{proposition}

\begin{proof}
Let $\mathcal T$ be the operator given by Proposition \ref{prop1}, namely such that $\mathcal Th$ is the unique solution $\phi$ of problem \eqref{eq. projected}; in other words, $\mathcal T$ is the inverse of $\mathcal L$ on the orthogonal complement of the kernel $\mathbf K$ given by \eqref{k}. Problem \eqref{eq. projected} is equivalent to the equation
$$\phi=\mathcal T(\mathcal E-\mathcal N(\phi)):=\mathcal A(\phi).$$
We prove that $\mathcal A$ is a contraction mapping on the ball
$$\mathbf B_M:=\left\{\phi\in L^\infty(\Omega_\varepsilon):\|\phi\|_\infty\le M\varepsilon^{\min\{\alpha,1\}}\log\frac1\varepsilon\right\},$$
provided $M$ is chosen large enough.\\
On one hand, using Proposition \ref{le. estimate of E} and \ref{prop1} and Lemma \ref{le. estimate of N}, for any $\phi\in\mathbf B_M$ we get
\begin{align*}
\|\mathcal A(\phi)\|_\infty\le&C\log\left(\frac1\varepsilon\right)(\|\mathcal E\|_\ast+\|\mathcal N (\phi)\|_\ast)\\
\le&C\log\left(\frac1\varepsilon\right)\left(\varepsilon^{\min\{\alpha,1\}}+\|\phi\|_\infty^2\right)\\
\le&C\varepsilon^{\min\{\alpha,1\}}\log\left(\frac1\varepsilon\right)\left(1+M^2\varepsilon^{\min\{\alpha,1\}}\log^2\frac1\varepsilon\right)\\
\le&M\varepsilon^{\min\{\alpha,1\}}\log\frac1\varepsilon,
\end{align*}
by choosing $M=2C$ and $\varepsilon$ small enough.\\
On the other hand, for any $\phi_1,\phi_2\in\mathbf B_M$, Lemma \ref{le. estimate of N} and Proposition \ref{prop1} yield
\begin{align*}
\|\mathcal A(\phi_1)-\mathcal A(\phi_2)\|_\infty\le&C\log\left(\frac1\varepsilon\right)\|\mathcal N(\phi_1)-\mathcal N(\phi_2)\|_\ast\\
\le&C\log\left(\frac1\varepsilon\right)(\|\phi_1\|_\infty+\|\phi_1\|_\infty)\|\phi_1-\phi_2\|_\infty\\
\le&CM\varepsilon^{\min\{\alpha,1\}}\log^2\left(\frac1\varepsilon\right)\|\phi_1-\phi_2\|_\infty\\
\le&\frac12\|\phi_1-\phi_2\|_\infty,
\end{align*}
for small $\varepsilon$.\\
Therefore the map $\mathcal A$ is a contraction map on $\mathbf B_M$, hence, by Banach's fixed point theorem, there exists a unique $\phi\in\mathbf B_M$ such that $\phi=\mathcal A(\phi)$.
\end{proof}\

\section{Finite-dimensional reduction}\

In the final section we will give conditions on $\xi$ in order to get $c_{i,j}=0$ in problem \eqref{eq. projected}, hence a true solution to \eqref{probphi} and \eqref{eq. original problem}.\\
We basically need to multiply equation \eqref{eq. projected} by each $z_{i,j}$ and integrate.

\begin{lemma}\label{le. c=0}
Let be $\phi$ the unique solution to problem \eqref{eq. projected} and, as in \eqref{zij},
$$z_{i,j}(x):=(4-\alpha)\mu_j\frac{x_i-\xi_{i,j}}{\mu_j^2+\left|x-\xi_j\right|^2}.$$
If
\begin{equation}\label{integralcij}
\int_{\Omega_\varepsilon}(\mathcal E-\mathcal L\phi-\mathcal N(\phi))z_{i,j}=0\qquad\forall\,i=1,2,\,j=1,\dots,m,
\end{equation}
then
$$c_{i,j}=0,\quad\forall\,i=1,2,\,j=1,\dots,m.$$
\end{lemma}

\begin{proof} We fix $k\in\{1,2\},l\in\{1,\dots,m\}$, multiply \eqref{eq. projected} by $z_{k,l}$ and integrate. Since $z_{k,l}$ and $\chi_jz_{i,j}$ are orthogonal if $(k,l)\ne(i,j)$, then
$$\int_{\Omega_\varepsilon}(\mathcal E-\mathcal L\phi-\mathcal N(\phi))z_{k,l}=\sum_{i=1}^2\sum_{j=1}^mc_{i,j}\int_{\Omega_\varepsilon}\chi_jz_{i,j}z_{k,l}=c_{k,l}\int_{\Omega_\varepsilon}\chi_lz_{k,l}^2.$$
From \eqref{integralckl} we deduce that if the left-hand side is $0$ one must have $c_{k,l}=0$.
\end{proof}\

Next, we show that the contributions from $\mathcal L\phi$ and $\mathcal N(\phi)$ in \eqref{integralcij} are negligible, provided $\alpha$ is in the range stated in Theorem \ref{th.1}.

\begin{lemma}\label{estimateLN}
Let $\phi$ be the unique solution to problem \eqref{eq. projected}. Then, for any $i=1,2,j=1,\dots,m$,
$$\left|\int_{\Omega_\varepsilon}(\mathcal L\phi)z_{i,j}\right|+\left|\int_{\Omega_\varepsilon}\mathcal N(\phi)z_{i,j}\right|=O\left(\varepsilon^{\min\{2\alpha,2\}}\log^2\frac1\varepsilon\right).$$
\end{lemma}

\begin{proof}
We start with the second term. From Lemma \ref{le. estimate of N} and Proposition \ref{le. unique solu} we get
$$\|\mathcal N(\phi)\|_\ast=O\left(\|\phi\|_\infty^2\right)=O\left(\varepsilon^{\min\{2\alpha,2\}}\log^2\frac1\varepsilon\right),$$
hence
\begin{equation}\label{estimateN}
\left|\int_{\Omega_\varepsilon}\mathcal N(\phi)z_{i,j}\right|\le C\|\mathcal N(\phi)\|_\ast\|z_{i,j}\|_\infty\le C\varepsilon^{\min\{2\alpha,2\}}\log^2\frac1\varepsilon.
\end{equation}
In order to handle the first term, we split $z_{i,j}=(1-\eta_{2,j})z_{i,j}+\eta_{2,j}z_{i,j}$, where $\eta_{2,j}$ is as in \eqref{eta12}.\\
From \eqref{eq. prop1 5}, we can estimate the first term as
$$\left|\int_{\Omega_\varepsilon}(\mathcal L\phi)\eta_{2,j}z_{i,j}\right|\le C\varepsilon^{\min\{\alpha,1\}}\|\phi\|_\infty\le C\varepsilon^{\min\{2\alpha,2\}}\log\frac1\varepsilon.$$
Moreover, since $0\le1-\eta_{2,j}\le\chi_{\left\{\left|x-\xi'_j\right|\ge\frac\tau{3\varepsilon}\right\}}(x)$, then
$$\|(1-\eta_{2,j})z_{i,j}\|_\infty\le\sup_{\left\{\left|x-\xi'_j\right|\ge\frac\tau{3\varepsilon}\right\}}|z_{i,j}(x)|\le C\varepsilon;$$
finally, since $\phi$ solves \eqref{eq. projected}, then, from \eqref{ckl},
$$\|\mathcal L\phi\|_\ast\le\|\mathcal E\|_\ast+\|\mathcal N(\phi)\|_\ast+\sum_{k=1}^2\sum_{l=1}^k|c_{k,l}|\|\chi_jz_{i,j}\|_\ast\le C\left(\|\mathcal E\|_\ast+\|\mathcal N(\phi)\|_\ast\right)\le C\left(\varepsilon^{\min\{\alpha,1\}}\log\frac1\varepsilon\right),$$
hence
$$\left|\int_{\Omega_\varepsilon}(\mathcal L\phi)z_{i,j}\right|\le\left|\int_{\Omega_\varepsilon}(\mathcal L\phi)\eta_{2,j}z_{i,j}\right|+\|\mathcal L\phi\|_\ast\|(1-\eta_{2,j})z_{i,j}\|_\infty\le C\varepsilon^{\min\{2\alpha,2\}}\log\frac1\varepsilon,$$
which, together with \eqref{estimateN}, completes the proof.
\end{proof}\

We then need the fundamental expansion of the energy functional, which highlights the role of the reduce functional $\Phi_m$:

\begin{proposition}\label{mainorder}
The following estimate holds true, for any $i=1,2,\,j=1,\dots,m$:
\begin{align*}
&\int_{\Omega_\varepsilon}\mathcal Ez_{i,j}\\
=&\left(-2(4-\alpha)\pi\frac{\partial_{\xi_i}k(\xi_j)}{k(\xi_j)}-\frac{(4-\alpha)^2}4\pi\partial_{\xi_{i,j}}\left(\sum_{k=1}^mH(\xi_k,\xi_k)+\sum_{k\ne l}G(\xi_k,\xi_l)\right)\right)\varepsilon+O\left(\varepsilon^{\min\{1+\alpha,2\}}\log^2\frac1\varepsilon\right)\\
=&\left(-\frac{(4-\alpha)^2\pi}4\partial_{\xi_{i,j}}\Phi_m(\xi)\right)\varepsilon+O\left(\varepsilon^{\min\{1+\alpha,2\}}\log^2\frac1\varepsilon\right),
\end{align*}
where $\Phi_m$ is the reduced functional given by \eqref{phi}.
\end{proposition}

\begin{proof}
We split the integral as
\begin{align*}
\int_{\Omega_\varepsilon}\mathcal Ez_{i,j}=&\sum_{k=1}^m\int_{\Omega_\varepsilon}\left(\int_{\Omega_\varepsilon}\frac{e^{w_k(y)}}{|x-y|^\alpha}dy\right)e^{w_k(x)}z_{i,j}(x)dx-\int_{\Omega_\varepsilon}\left(\int_{\Omega_\varepsilon}\frac{k(\varepsilon y)e^{V(y)}}{|x-y|^\alpha}dy\right)k(\varepsilon x)e^{V(x)}z_{i,j}(x)dx\\
=&\underbrace{\int_{B_j\times B_j}\frac{e^{w_j(x)}e^{w_j(y)}-k(\varepsilon x)e^{V(x)}k(\varepsilon y)e^{V(y)}}{|x-y|^\alpha}z_{i,j}(x)dxdy}_{=:J_1}\\
=&\sum_{k\ne j}\underbrace{\int_{B_k\times B_k}\frac{e^{w_k(x)}e^{w_k(y)}-k(\varepsilon x)e^{V(x)}k(\varepsilon y)e^{V(y)}}{|x-y|^\alpha}z_{i,j}(x)dxdy}_{=:J_{2,k}}\\
&+\sum_{k=1}^m\underbrace{\int_{(\Omega_\varepsilon\times\Omega_\varepsilon)\setminus B_k\times B_k}\frac{e^{w_k(x)}e^{w_k(y)}}{|x-y|^\alpha}z_{i,j}(x)dxdy}_{=:J_{3,k}}\\
&-\underbrace{\int_{(\Omega_\varepsilon\times\Omega_\varepsilon)\setminus\bigcup_{k,l=1}^m(B_k\times B_l)}\frac{k(\varepsilon x)e^{V(x)}k(\varepsilon y)e^{V(y)}}{|x-y|^\alpha}z_{i,j}(x)dxdy}_{=:J_4}\\
&-\sum_{k\ne j}\underbrace{\int_{B_j\times B_k}\frac{k(\varepsilon x)e^{V(x)}k(\varepsilon y)e^{V(y)}}{|x-y|^\alpha}z_{i,j}(x)dxdy}_{=:J_{5,k}}\\
&-\sum_{k\ne j,l\ne k}\underbrace{\int_{B_k\times B_l}\frac{k(\varepsilon x)e^{V(x)}k(\varepsilon y)e^{V(y)}}{|x-y|^\alpha}z_{i,j}(x)dxdy}_{=:J_{6,k,l}},	
\end{align*}
where $B_k$ is defined as in \eqref{bj}.\\
From \eqref{eq. in ball} we deduce that, for $x,y\in B_j$,
\begin{align*}
&e^{w_j(x)}e^{w_j(y)}-k(\varepsilon x)e^{V(x)}k(\varepsilon y)e^{V(y)}\\
=&e^{w_j(x)}e^{w_j(y)}\left(\varepsilon\Psi_j\cdot\left(x-\xi'_j\right)+\varepsilon\Psi_j\cdot\left(y-\xi'_j\right)+O\left(\varepsilon^2\left(\left|x-\xi'_j\right|^2+\left|y-\xi'_j\right|^2\right)\right)\right),
\end{align*}
where
$$\Psi_j:=-\frac{\nabla k(\xi_j)}{k\left(\xi_j\right)}-\frac{4-\alpha}4\left(\nabla H(\xi_j,\xi_j)+\sum_{k\ne j}\nabla G(\xi_j,\xi_k)\right);$$
we remark that, due to the symmetry of $H(x,y),G(x,y)$ with respect to the variables, it holds
$$\Psi_j=\frac{4-\alpha}8\partial_{\xi_j}\Phi_m(\xi).$$
Therefore, exploiting the cancellations by symmetry, we get
\begin{align*}
J_1=&\varepsilon\int_{B_j\times B_j}\frac{e^{w_j(x)}e^{w_j(y)}}{|x-y|^\alpha}z_{i,j}(x)\left(\Psi_j\cdot\left(x-\xi'_j\right)+\Psi_j\cdot\left(y-\xi'_j\right)\right)dxdy\\
&+O\left(\varepsilon^2\int_{B_j\times B_j}\frac{e^{w_j(x)}e^{w_j(y)}}{|x-y|^\alpha}z_{i,j}(x)\left(\left|x-\xi'_j\right|^2+\left|y-\xi'_j\right|^2\right)dxdy\right)\\
=&\varepsilon\Psi_{i,j}\underbrace{\int_{B_j\times B_j}\frac{e^{w_j(x)}e^{w_j(y)}}{|x-y|^\alpha}z_{i,j}(x)\left(\left(x_i-\xi'_{i,j}\right)+\left(y_i-\xi'_{i,j}\right)\right)dxdy}_{=:J_1'}\\
&+O\left(\varepsilon^2\underbrace{\int_{B_j\times B_j}\frac1{|x-y|^\alpha\left(1+\left|x-\xi'_j\right|\right)^{3-\alpha}\left(1+\left|y-\xi'_j\right|\right)^{2-\alpha}}dxdy}_{=:J_1''}\right);
\end{align*}
the last term is indeed negligible because, arguing as in \eqref{i1j-i2j}, we get
\begin{align*}
J_1''=&O\left(\int_{\left\{|x-y|\le\frac{1+\left|x-\xi'_j\right|}2\right\}}\frac1{|x-y|^\alpha\left(1+\left|x-\xi'_j\right|\right)^{5-2\alpha}}dxdy\right.\\
&\left.+\int_{\left\{\frac{1+\left|x-\xi'_j\right|}2<|x-y|\le2\left(1+\left|x-\xi'_j\right|\right)\right\}}\frac1{\left(1+\left|x-\xi'_j\right|\right)^3\left(1+\left|y-\xi'_j\right|\right)^{2-\alpha}}dxdy\right.\\
&\left.+\int_{\left\{|x-y|>2\left(1+\left|x-\xi'_j\right|\right)\right\}}\frac1{\left(1+\left|x-\xi'_j\right|\right)^{3-\alpha}\left(1+\left|y-\xi'_j\right|\right)^2}dxdy\right)\\
=&O\left(\int_{\left\{\left|x-\xi'_j\right|\le\frac\tau\varepsilon\right\}}\left(\frac{1+\log\left(1+\left|x-\xi'_j\right|\right)}{\left(1+\left|x-\xi'_j\right|\right)^{3-\alpha}}\right)dx\right)\\
=&O\left(\varepsilon^{\min\{\alpha-1,0\}}\log^2\frac1\varepsilon\right).
\end{align*}
To estimate $J_1'$, we recall that, since $w_j$ solves the limiting problem \eqref{eq. limit problem}, then, for any $x\in\mathbb R^2$,
\begin{equation}\label{eqbubble}
\frac{2(4-\alpha)\mu_j^2}{\left(\mu_j^2+\left|x-\xi'_j\right|^2\right)^2}=-\Delta w_j(x)=\left(\int_{\mathbb R^2}\frac{e^{w_j(y)}}{|x-y|^\alpha}dy\right)e^{w_j(x)};
\end{equation}
similarly, since $z_{i,j}$ is in the kernel of the limiting operator $\mathcal L_j$ defined in \eqref{lj}, then
\begin{align*}
\frac{8\mu_j^2}{\left(\mu_j^2+\left|x-\xi'_j\right|^2\right)^2}z_{i,j}(x)=&-\Delta z_{i,j}(x)\\
=&\left(\int_{\mathbb R^2}\frac{e^{w_j(y)}z_{i,j}(y)}{|x-y|^\alpha}dy\right)e^{w_j(x)}+\left(\int_{\mathbb R^2}\frac{e^{w_j(y)}}{|x-y|^\alpha}dy\right)e^{w_j(x)}z_{i,j}(x);
\end{align*}
Putting together with \eqref{eqbubble}, we get:
\begin{align*}
\left(\int_{\mathbb R^2}\frac{e^{w_j(y)}}{|x-y|^\alpha}dy\right)e^{w_j(x)}z_{i,j}(x)=&\frac{2(4-\alpha)\mu_j^2}{\left(\mu_j^2+\left|x-\xi'_j\right|^2\right)^2}z_{i,j}(x),\\
\left(\int_{\mathbb R^2}\frac{e^{w_j(y)}z_{i,j}(y)}{|x-y|^\alpha}dy\right)e^{w_j(x)}=&\frac{2\alpha\mu_j^2}{\left(\mu_j^2+\left|x-\xi'_j\right|^2\right)^2}z_{i,j}(x);
\end{align*}
on the other hand, from \eqref{triang} we get
\begin{align*}
&\int_{\mathbb R^2\setminus B_j}\frac{e^{w_j(y)}\left(1+|z_{i,j}(y)|\left|y-\xi'_j\right|\right)}{|x-y|^\alpha}dy\\
=&O\left(\int_{\left\{\left|y-\xi'_j\right|>\frac\tau\varepsilon\right\}}\frac1{|x-y|^\alpha\left(1+\left|y-\xi'_j\right|\right)^{4-\alpha}}dy\right)\\
=&O\left(\int_{\left\{|y-x|\le\frac{1+\left|x-\xi'_j\right|}2\right\}}\frac1{|x-y|^\alpha\left(1+\left|x-\xi'_j\right|\right)^{4-\alpha}}dy\right.\\
&\left.+\int_{\left\{\frac{1+\left|x-\xi'_j\right|}2<|y-x|\le2\left(1+\left|x-\xi'_j\right|\right)\right\}}\frac1{\left(1+\left|x-\xi'_j\right|\right)^\alpha\left(1+\left|y-\xi'_j\right|\right)^{4-\alpha}}dy\right.\\
&\left.+\int_{\left\{\left|y-\xi'_j\right|>\frac\tau\varepsilon\right\}}\frac1{\left(1+\left|y-\xi'_j\right|\right)^4}dy\right)\\
=&O\left(\frac1{\left(1+\left|x-\xi'_j\right|\right)^2}+\varepsilon^2\right)\\
=&O\left(\varepsilon^2\right)
\end{align*}
and, similarly,
\begin{align*}
\int_{\mathbb R^2\setminus B_j}\frac{e^{w_j(y)}|z_{i,j}(y)|}{|x-y|^\alpha}dy=&O\left(\int_{\left\{\left|y-\xi'_j\right|>\frac\tau\varepsilon\right\}}\frac1{|x-y|^\alpha\left(1+\left|y-\xi'_j\right|\right)^{5-\alpha}}dy\right)=O\left(\varepsilon^3\right),\\
\int_{\mathbb R^2\setminus B_j}\frac{e^{w_j(y)}\left|y-\xi'_j\right|}{|x-y|^\alpha}dy=&O\left(\int_{\left\{\left|y-\xi'_j\right|>\frac\tau\varepsilon\right\}}\frac1{|x-y|^\alpha\left(1+\left|y-\xi'_j\right|\right)^{3-\alpha}}dy\right)=O(\varepsilon).\\
\end{align*}
Therefore,
\begin{align*}
J_1'=&\int_{\mathbb R^2}\left(\int_{\mathbb R^2}\frac{e^{w_j(y)}}{|x-y|^\alpha}dy\right)e^{w_j(x)}z_{i,j}(x)\left(x_i-\xi'_{i,j}\right)dx+\int_{\mathbb R^2}\left(\int_{\mathbb R^2}\frac{e^{w_j(x)}z_{i,j}(x)}{|x-y|^\alpha}dx\right)e^{w_j(y)}\left(y_i-\xi'_{i,j}\right)dy\\
&+O\left(\int_{B_j}\left(\int_{\mathbb R^2\setminus B_j}\frac{e^{w_j(y)}}{|x-y|^\alpha}dy\right)e^{w_j(x)}|z_{i,j}(x)|\left|x-\xi'_j\right|dx\right.\\
&\left.+\int_{\mathbb R^2}\left(\int_{\mathbb R^2\setminus B_j}\frac{e^{w_j(x)}|z_{i,j}(x)|\left|x-\xi'_j\right|}{|x-y|^\alpha}dx\right)e^{w_j(y)}dy\right.\\
&\left.+\int_{B_j}\left(\int_{\mathbb R^2\setminus B_j}\frac{e^{w_j(y)}|z_{i,j}(y)|}{|x-y|^\alpha}dy\right)e^{w_j(x)}\left|x-\xi'_j\right|dx\right.\\
&\left.+\int_{\mathbb R^2}\left(\int_{\mathbb R^2\setminus B_j}\frac{e^{w_j(x)}\left|x-\xi'_j\right|}{|x-y|^\alpha}dx\right)e^{w_j(y)}|z_{i,j}(y)|dy\right)\\
=&\int_{\mathbb R^2}\frac{8\mu_j^2}{\left(\mu_j^2+\left|x-\xi'_j\right|^2\right)^2}z_{i,j}(x)\left(x_i-\xi'_{i,j}\right)dx\\
&+O\left(\varepsilon^2\int_{\mathbb R^2}e^{w_j}dx+\varepsilon^3\int_{B_j}e^{w_j(x)}\left|x-\xi'_j\right|dx+\varepsilon\int_{\mathbb R^2}e^{w_j(y)}|z_{i,j}(y)|dy\right)\\
=&2(4-\alpha)\pi+O(\varepsilon);
\end{align*}
we then showed that
$$J_1=2(4-\alpha)\pi\Psi_{i,j}\varepsilon+O\left(\varepsilon^{\min\{1+\alpha,2\}}\log^2\frac1\varepsilon\right)=\left(-\frac{(4-\alpha)^2\pi}4\partial_{\xi_{i,j}}\Phi_m(\xi)\right)\varepsilon+O\left(\varepsilon^{\min\{1+\alpha,2\}}\log^2\frac1\varepsilon\right),$$
hence we are left with showing that $J_{2,k},J_{3,k},J_4,J_{5,k},J_{6,k,l}$ are all lower order terms.\\
To handle $J_{2,k}$, we use that, again from \eqref{eq. in ball}, if $x,y\in B_k$, then
$$e^{w_j(x)}e^{w_j(y)}-k(\varepsilon x)e^{V(x)}k(\varepsilon y)e^{V(y)}=O\left(\varepsilon\frac{\left|x-\xi'_k\right|+\left|y-\xi'_k\right|}{\left(1+\left|x-\xi'_k\right|\right)^{4-\alpha}\left(1+\left|y-\xi'_k\right|\right)^{4-\alpha}}\right);$$
moreover, if $k\ne j$, then $|z_{i,j}(x)|=O(\varepsilon)$, hence:
\begin{align*}
J_{2,k}=&O\left(\varepsilon^2\int_{B_k\times B_k}\frac1{|x-y|^\alpha\left(1+\left|x-\xi'_k\right|\right)^{3-\alpha}\left(1+\left|y-\xi'_k\right|\right)^{4-\alpha}}dxdy\right)\\
=&O\left(\varepsilon^2\int_{B_k}\frac1{\left(1+\left|x-\xi'_k\right|\right)^3}dx\right)\\
=&O\left(\varepsilon^2\right).
\end{align*}
Then, we recall that outside $B_k$ we have $e^{w_k}=O\left(\varepsilon^{4-\alpha}\right)$, hence:
$$J_{3,k}=O\left(\varepsilon^{4-\alpha}\int_{\Omega_\varepsilon\times\Omega_\varepsilon}\frac1{|x-y|^\alpha}\frac1{\left(1+\left|y-\xi'_k\right|\right)^{4-\alpha}}dxdy\right)=O\left(\varepsilon^{4-\alpha}\int_{\Omega_\varepsilon}\frac1{\left(1+\left|x-\xi'_k\right|\right)^\alpha}dx\right)=O\left(\varepsilon^2\right);$$
similarly, since $e^V=O\left(\varepsilon^{4-\alpha}\right)$ in $B_0$, from \eqref{eq. out ball} we get
$$J_4=O\left(\varepsilon^{4-\alpha}\sum_{k=1}^m\int_{\Omega_\varepsilon\times\Omega_\varepsilon}\frac1{|x-y|^\alpha\left(1+\left|y-\xi'_k\right|\right)^{4-\alpha}}dxdy\right)=O\left(\varepsilon^2\right).$$
As for $J_{5,k}$, we use \eqref{eq. in ball}, the fact that $|x-y|\ge\frac\tau\varepsilon$ for any $x\in B_k,y\in B_j$ and some cancellations due to symmetry:
\begin{align*}
J_{5,k}=&\int_{B_j\times B_k}\frac{e^{w_k(y)}e^{w_j(x)}\left(1+O\left(\varepsilon\left|x-\xi'_j\right|+\varepsilon\left|y-\xi'_k\right|\right)\right)}{|x-y|^\alpha}z_{i,j}(x)dxdy\\
=&\int_{B_j}\left(\int_{\mathbb R^2}\frac{e^{w_k(y)}}{|x-y|^\alpha}dy+O\left(\varepsilon^2\right)\right)e^{w_j(x)}z_{i,j}(x)dx\\
&+O\left(\varepsilon^{1+\alpha}\int_{B_j\times B_k}\frac1{\left(1+\left|x-\xi'_k\right|\right)^{3-\alpha}\left(1+\left|y-\xi'_j\right|\right)^{4-\alpha}}dxdy\right)\\
=&\int_{B_j}\frac{2\sqrt{\frac{\pi(4-\alpha)}{2-\alpha}}\mu_k^\frac\alpha2}{\left(\mu_k^2+\left|x-\xi'_k\right|^2\right)^{\frac\alpha2}}e^{w_j(x)}z_{i,j}(x)dx+O\left(\varepsilon^{1+\alpha}\right)\\
=&\int_{B_j}\left(\frac{2\sqrt{\frac{\pi(4-\alpha)}{2-\alpha}}\mu_k^\frac\alpha2}{\left|\xi'_j-\xi'_k\right|^\alpha}+O\left(\varepsilon\left|x-\xi'_j\right|+\varepsilon^2\right)\right)e^{w_j(x)}z_{i,j}(x)dx+O\left(\varepsilon^{1+\alpha}\right)\\
=&O\left(\varepsilon^{\min\{1+\alpha,2\}}\right);
\end{align*}
finally, using again the distance between the balls $B_k$ and $B_l$, \eqref{eq. out ball} and $|z_{i,j}|=O(\varepsilon)$ on $B_k$ for $k\ne j$,
$$J_{6,k,l}=O\left(\varepsilon^{1+\alpha}\int_{B_k\times B_l}\frac1{\left(1+\left|x-\xi'_k\right|\right)^{4-\alpha}\left(1+\left|y-\xi'_l\right|\right)^{4-\alpha}}dxdy\right)=O\left(\varepsilon^{1+\alpha}\right),$$
which concludes the proof.
\end{proof}\

\begin{proof}[Proof of Theorem \ref{th.1}] From Proposition \ref{le. unique solu} we get a solution to problem \eqref{eq. projected}, hence we suffice to get $c_{i,j}=0$ for all $i=1,2$ and $j=1,\dots,m$. Moreover, since $\alpha>\frac12$, from Lemmas \ref{le. c=0} and \ref{estimateLN}, we deduce that $\int_{\Omega_\varepsilon}\mathcal Ez_{i,j}=0$ implies, for $\varepsilon>0$ small enough, $c_{i,j}=0$. Therefore, by Proposition \ref{mainorder}, if the functional $\Phi_m(\xi)$ has a stable critical point, then we obtain a solution to \eqref{eq. original problem}.\\
From the explicit definition of $U$ and the fact that $\|\phi\|_\infty\underset{\varepsilon\to0}\to0$ we easily get the qualitative properties as stated in Theorem \ref{th.1}.
\end{proof}\

\begin{proof}[Proof of Theorem \ref{m=1}]
If $m=1$, then Proposition \ref{le. estimate of E} reads as $\|\mathcal E\|_\ast\le C\varepsilon$, since the terms \eqref{knotj1} and \eqref{knotj2} of order $\varepsilon^\alpha$ only appear if $m\ge2$. Therefore, as is clear from the proof, in Proposition \ref{le. unique solu} one gets $\|\phi\|_\infty\le M\varepsilon\log\frac1\varepsilon$, hence the estimate in Lemma \ref{estimateLN} is of order $\varepsilon^2\log^2\frac1\varepsilon$; such terms are thus always negligible in the expansion in Proposition \ref{mainorder} and we may conclude as in the proof of Theorem \ref{th.1}.\\
Finally, we show that $\Phi_1$ has a stable critical point, yielding a solution to problem \eqref{eq. original problem}: since $H(\xi,\xi)\underset{\xi\to\partial\Omega}\to-\infty$, then the same holds true for $\Phi_1(\xi)=\frac8{4-\alpha}\log k(\xi)+H(\xi,\xi)$, hence $\Phi_1(\xi)$ attains its global maximum at some point of $\Omega$, which is necessarily a stable critical point.
\end{proof}\

\section*{Acknowledgments}
The first author is supported by MUR-PRIN-2022AKSNE ``Variational and Analytical  aspects of Geometric PDEs'' and by the INdAM-GNAMPA project, CUP E53C25002010001 ``Punti critici e concentrazioni di soluzioni di PDEs''.\\
The second author is supported by Southwest University graduate research innovation project (NO. SWUB25029).

\end{document}